\documentclass[11pt,fleqn,mathserif]{article}
\usepackage{stmaryrd}
\usepackage{times}
\usepackage[top=2.5cm,bottom=2cm,left=2.5cm,right=2.5cm]{geometry}
\usepackage{amsmath,amssymb,mathrsfs,amsthm}        
\usepackage[linkcolor=blue,pdfstartview=FitH,
            CJKbookmarks=true,bookmarksnumbered=true,bookmarksopen=true,
            colorlinks=true, 
            pdfborder=000,   
            citecolor=blue]{hyperref}
\usepackage{graphicx,graphics,subfigure,float,caption2,booktabs}
\usepackage{varioref}
\newtheorem{theorem}{Theorem}
\newtheorem{lemma}{Lemma}

\newtheorem{example}{Example}

\numberwithin{equation}{section}
\begin{document}

\title{Numerical scheme for the Fokker-Planck equations describing anomalous diffusions with two internal states}
\author{Daxin Nie,~~Jing Sun,~~Weihua Deng\footnote{Corresponding Author. E-mail: dengwh@lzu.edu.cn}\\[10pt]
        {School of Mathematics and Statistics, Lanzhou University, Lanzhou 730000, P.R. China}
       }

\date{}
\maketitle
\begin{abstract}
Recently, the fractional Fokker-Planck equations (FFPEs) with multiple internal states are built for the particles undergoing anomalous diffusion with different waiting time distributions for different internal states, which describe the distribution of positions of the particles [Xu and Deng, Math. Model. Nat. Phenom., $\mathbf{13}$, 10 (2018)]. In this paper, we first develop the Sobolev regularity of the FFPEs with two internal states, including the homogeneous problem with smooth and nonsmooth initial values and the inhomogeneous problem with vanishing initial value, and then we design the numerical scheme for the system of fractional partial differential equations based on the finite element method for the space derivatives and convolution quadrature for the time fractional derivatives. The optimal error estimates of the scheme under the above three different conditions are provided for both space semidiscrete and fully discrete schemes. Finally, one- and two-dimensional numerical experiments are performed to confirm our theoretical analysis and the predicted convergence order.


\end{abstract}
\begin{abstract}
Recently, the fractional Fokker-Planck equations (FFPEs) with multiple internal states are built for the particles undergoing anomalous diffusion with different waiting time distributions for different internal states, which describe the distribution of positions of the particles [Xu and Deng, Math. Model. Nat. Phenom., $\mathbf{13}$, 10 (2018)]. In this paper, we first develop the Sobolev regularity of the FFPEs with two internal states, including the homogeneous problem with smooth and nonsmooth initial values and the inhomogeneous problem with vanishing initial value, and then we design the numerical scheme for the system of fractional partial differential equations based on the finite element method for the space derivatives and convolution quadrature for the time fractional derivatives. The optimal error estimates of the scheme under the above three different conditions are provided for both space semidiscrete and fully discrete schemes. Finally, one- and two-dimensional numerical experiments are performed to confirm our theoretical analysis and the predicted convergence order.
\end{abstract}

\maketitle
\section{Introduction}
The Forkker-Planck equation (FPE) is one of the most important equations of statistical physics, which describe the time evolution of the probability density function (PDF) of positions of  particles. With the rapid development of technologies, the colorful anomalous diffusion phenomena are observed. At the early stage, the fractional Forkker-Planck equations (FFPEs) were derived to model the anomalous physical processes with power-law waiting time and/or jump length distribution(s) \cite{Barkai2001,Barkai2000}. However, the solutions of the most of the FFPEs can't be obtained explicitly and this fact motivates many authors to develop the effective numerical methods for FFPEs \cite{Deng2007,Deng2009,Heinsalu2006,Liu2004,Meerschaert2006}.


With the deep insight on the mechanism of anomalous diffusion, in some cases, the concept of internal states has to be introduced for more accurately modeling the real natural phenomena. Specifying each internal state with particular waiting time and jump length distributions and introducing a Markov chain with its transition matrix deciding the transition of the internal states, Ref. \cite{Xu2018}, recently, builds the multiple-internal-states FFPEs (see \cite{Xu2018-2} for the multiple-internal-states L\'evy walk). Efficiently solving the model naturally becomes an urgent topic.
In this paper, we provide a numerical scheme and do the numerical analyses for the FFPEs with two internal states \cite{Xu2018}, i.e.,
\begin{equation}\label{equmatrixequ}
\left \{
\begin{split}
&\mathbf{M}^T\frac{\partial }{\partial t} \mathbf{G}=(\mathbf{M}^T-\mathbf{I}){\rm diag}(~_0D^{1-\alpha_1}_t, ~_0D^{1-\alpha_2}_t)\mathbf{G}\\
&\quad\quad\quad\quad\quad\quad +\mathbf{M}^T{\rm diag}(~_0D^{1-\alpha_1}_t, ~_0D^{1-\alpha_2}_t)\Delta\mathbf{G}+\mathbf{M}^T\mathbf{F} \quad {\rm in}\ \Omega,\ t\in[0,T],\\
&\mathbf{G}(\cdot,0)=\mathbf{G}_0 \quad\quad\quad\quad\quad\quad\quad\quad\quad\quad\quad\quad\quad\quad\quad\quad\quad\quad\quad\quad {\rm in}\ \Omega,\\
&\mathbf{G}=0 \quad\quad\quad\quad\quad\quad\quad\quad\quad\quad\quad\quad\quad\quad\quad\quad\quad\quad\quad\quad\quad\quad\ \ \, {\rm on}\ \partial\Omega,\ t\in[0,T],
\end{split}
\right .
\end{equation}
where $\Omega$ denotes a bounded convex polygonal domain in $R^d$ $(d=1,2,3)$; $\mathbf{M}$ is the transition matrix of a Markov chain, being a $2\times 2$ invertible matrix here; $\mathbf{G}=[G_1,G_2]^T$ denotes the solution of the system \eqref{equmatrixequ} and $\mathbf{F}=[f_1,f_2]^T$ is the source term; $\mathbf{G}_0=[G_{1,0},G_{2,0}]^T$ is the initial value; $\mathbf{I}$ is an identity matrix; `diag' denotes a diagonal matrix formed from its vector argument,  and $~_0D^{1-\alpha_{i}}_t$, $i=1,2$ are the Riemann-Liouville fractional derivatives defined by \cite{Podlubny1999}
\begin{equation}
_{0}D^{1-\alpha_{i}}_tG=\frac{1}{\Gamma(\alpha_{i})}\frac{\partial}{\partial t}\int^t_{0}(t-\xi)^{\alpha_{i}-1}G(\xi)d\xi, ~\alpha_{i}\in(0,1), ~i=1,2.
\end{equation}

It can be noted that the system \eqref{equmatrixequ} is constituted of fractional derivatives in time and Laplacian operator in space. Numerical methods for the time fractional derivatives have gained widespread concerns  \cite{Gao2014,Langlands2005,Li2010,Li2014,Lin2007,Zeng2013}. And in recent years, convolution quadrature introduced in \cite{Lubich1988-1,Lubich1988-2,Lubich2004} has been widely used in discretizing the time fractional derivative operators \cite{Jin2013,Jin2014,Jin2015,Jin2015-2,Jin2016,Lubich1996}, of which the main advantage is that it does't need the assumption on the regularity of the solution and a higher order one can be obtained after some suitable modifications.
It seems that the theoretical analysis and numerical simulation for the system of fractional partial differential equations are scare. Here, we try to fill the gap and provide the Sobolev regularity of solutions for the system \eqref{equmatrixequ}, i.e., we obtain the solutions $G_1(t),G_2(t)\in H^1_0(\Omega)\bigcap H^2(\Omega)$ for both smooth initial values $G_{1,0},G_{2,0}\in H^1_0(\Omega)\bigcap H^2(\Omega)$ and nonsmooth initial values $G_{1,0},G_{2,0}\in L^2(\Omega)$ for the homogeneous problem; see Theorem \ref{thmhomoregularity}. For the inhomogeneous problem, we prove that the solutions $G_1(t),G_2(t)\in H^1_0(\Omega)\bigcap H^2(\Omega)$ when $f_1,f_2\in L^2(\Omega)$ and $G_{1,0}=0$, $G_{2,0}=0$ in Theorem \ref{thminhomoregularity}. Furthermore, we use the convolution quadrature to discretize the time fractional derivatives and finite element method for the space operators, and then we give a complete theoretical analysis for the scheme under three different initial conditions. At last, numerical results for one- and two-dimensional examples are presented to illustrate the effectiveness of the numerical scheme.

The paper is organized as follows. In Section 2, we first introduce the notations and then focus on  the Sobolev regularity of the solutions for the homogeneous problem \eqref{equmatrixequ} with smooth and nonsmooth initial values and inhomogeneous problem \eqref{equmatrixequ} with vanishing initial value. In Section 3, we do the space discretization by the finite element method and provide the error estimates for the semidiscrete scheme under three different initial conditions. In Section 4, we use the convolution quadrature to discretize the time fractional derivatives and provide error estimates for the fully discrete scheme. In the last section, we confirm the  theoretically predicted convergence order by the one- and two-dimensional numerical examples. Throughout this paper, $C$ denotes a generic positive constant, whose value may differ at each occurrence.

\section{Regularity of the solution}
\subsection{Preliminaries}
 We first introduce some notations. Denote $G_1(t)$, $G_2(t)$, $f_1(t)$, and $f_2(t)$ as the functions $G_1(\cdot,t)$, $G_2(\cdot,t)$, $f_1(\cdot,t)$, and $f_2(\cdot,t)$ respectively.  Let $A=-\Delta:H^1_0(\Omega)\bigcap H^2(\Omega) \rightarrow L^2(\Omega)$ be the negative Laplacian operator with a zero Dirichlet boundary condition and $ {(\lambda_j,\varphi_j)} $ be its eigenvalues ordered non-decreasingly and the corresponding eigenfunctions normalized in the $ L^2(\Omega) $ norm.  For any $ r\geq  0 $, denote the space $ \dot{H}^r(\Omega)=\{v\in L^2(\Omega): A^{\frac{r}{2}}v\in L^2(\Omega) \}$ with the norm \cite{Thomee2006}
	\begin{equation*}
		\|v\|^2_{\dot{H}^r(\Omega)}=\sum_{j=1}^{\infty}\lambda_j^r(v,\varphi_j)^2.
	\end{equation*}
	Thus $ \dot{H}^0(\Omega)=L^2(\Omega) $, $\dot{H}^1(\Omega)=H^1_0(\Omega)$, and $\dot{H}^2(\Omega)=H^2(\Omega)\bigcap H^1_0(\Omega)$.
We denote $\|\cdot\|$ as the operator norm from $L^2(\Omega)$ to $L^2(\Omega)$, and use the notation `$\tilde{~}$' for taking Laplace transform.
	
Furthermore, for $\kappa>0$ and $\pi/2<\theta<\pi$, we denote sector $\Sigma_{\theta}$ and $\Sigma_{\theta,\kappa}$ as
    \begin{equation*}
        \begin{aligned}
        &\Sigma_{\theta}=\{z\in\mathbb{C}:z\neq 0,|\arg z|\leq \theta\},\\
        &\Sigma_{\theta,\kappa}=\{z\in\mathbb{C}:|z|>\kappa,|\arg z|\leq \theta\},\\
        \end{aligned}
    \end{equation*}
 and define the contour $\Gamma_{\theta,\kappa}$ by
    \begin{equation*}
    \Gamma_{\theta,\kappa}=\{r e^{-i\theta}: r\geq \kappa\}\bigcup\{\kappa e^{i\psi}: |\psi|\leq \theta\}\bigcup\{r e^{i\theta}: r\geq \kappa\},
    \end{equation*}
    where the circular arc is oriented counterclockwise and the two rays are oriented with an increasing imaginary part.

\subsection{A priori estimate of the solution}
According to the property of the transition matrix of a Markov chain \cite{Xu2018}, we can denote the matrix $\mathbf{M}$ as
\begin{equation*}
	\mathbf{M}=\left [\begin{matrix}
	m&1-m\\
	1-m&m
	\end{matrix}\right ],\quad m\in[0,1],
\end{equation*}
and the fact that matrix $\mathbf{M}$ is invertible leads to
\begin{equation*}
(\mathbf{M}^T)^{-1}=\left [\begin{matrix}
\frac{m}{2m-1}&\frac{m-1}{2m-1}\\
\frac{m-1}{2m-1}&\frac{m}{2m-1}
\end{matrix}\right ].
\end{equation*}
Then the system \eqref{equmatrixequ} can be rewritten as
\begin{equation}\label{equrqtosol}
	\left \{
	\begin{aligned}
	&\frac{\partial G_1}{\partial t}+a~_0D^{1-\alpha_1}_tG_1-~_0D^{1-\alpha_1}_t\Delta G_1=a~_0D^{1-\alpha_2}_tG_2+f_1\quad\quad\quad\,\, {\rm in}\ \Omega,\ t\in[0,T],\\
	&\frac{\partial G_2}{\partial t}+a~_0D^{1-\alpha_2}_tG_2-~_0D^{1-\alpha_2}_t\Delta G_2=a~_0D^{1-\alpha_1}_tG_1+f_2   \quad\quad\,\,\quad {\rm in}\ \Omega,\ t\in[0,T],\\
	&\mathbf{G}(\cdot,0)=\mathbf{G}_0 \quad\quad\quad\quad\quad\quad\quad\quad\quad\quad\,\quad\quad\quad\quad\quad\quad\quad\quad\quad\quad\quad\quad {\rm in}\ \Omega,\\
	&\mathbf{G}=0 \quad\quad\quad\quad\quad\quad\quad\quad\quad\quad\quad\quad\quad\quad\quad\quad\quad\quad\quad\quad\quad\quad\quad\quad\quad {\rm on}\ \partial\Omega,\ t\in[0,T],
	\end{aligned}
	\right .
\end{equation}
where $ a=\frac{1-m}{2m-1} $.
Taking the Laplace transforms for the first two equations of the system \eqref{equrqtosol} and using the identity $\widetilde{~_0D^{\alpha}_tu}(z)=z^\alpha\tilde{u}(z)$ \cite{Podlubny1999}, we have
\begin{equation}\label{equequinlapform}
\begin{aligned}
	&z\tilde{G}_1+az^{1-\alpha_1}\tilde{G}_1+z^{1-\alpha_1}A\tilde{G}_1=az^{1-\alpha_2}\tilde{G}_2+\tilde{f}_1+G_{1,0},\\
	&z\tilde{G}_2+az^{1-\alpha_2}\tilde{G}_2+z^{1-\alpha_2}A\tilde{G}_2=az^{1-\alpha_1}\tilde{G}_1+\tilde{f}_2+G_{2,0}.
\end{aligned}
\end{equation}
Simple calculation leads to

\begin{equation*}
	\begin{aligned}
		&\tilde{G}_1=(z+az^{1-\alpha_1}+z^{1-\alpha_1}A)^{-1}\left(az^{1-\alpha_2}\tilde{G}_2+\tilde{f}_1+G_{1,0}\right),\\
		&\tilde{G}_2=(z+az^{1-\alpha_2}+z^{1-\alpha_2}A)^{-1}\left(az^{1-\alpha_1}\tilde{G}_1+\tilde{f}_2+G_{2,0}\right).
	\end{aligned}
\end{equation*}
Then we obtain
\begin{equation*}
	\begin{aligned}
		\tilde{G}_1=&H(z)\left ((z^{\alpha_2}+a+A)z^{\alpha_1-1}\tilde{f}_1+az^{\alpha_1-1}\tilde{f}_2\right )\\
		&+H(z)\left ((z^{\alpha_2}+a+A)z^{\alpha_1-1}G_{1,0}+az^{\alpha_1-1}G_{2,0}\right ),\\
		\tilde{G}_2=&H(z)\left (az^{\alpha_2-1}\tilde{f}_1+(z^{\alpha_1}+a+A)z^{\alpha_2-1}\tilde{f}_2\right )\\
		&+H(z)\left (az^{\alpha_2-1}G_{1,0}+(z^{\alpha_1}+a+A)z^{\alpha_2-1}G_{2,0}\right ),
	\end{aligned}
\end{equation*}
where
\begin{equation}\label{equdefHz}
	H(z)=\left((z^{\alpha_1}+a+A)(z^{\alpha_2}+a+A)-a^2\right)^{-1}.
\end{equation}
Introducing
\begin{equation}\label{H12}
\begin{aligned}
H_{\alpha_1}(z)=H(z)(z^{\alpha_2}+a+A),\quad
H_{\alpha_2}(z)=H(z)(z^{\alpha_1}+a+A),
\end{aligned}
\end{equation}
there exists
\begin{equation}\label{equsolformofLaplace}
\begin{aligned}
\tilde{G}_1=&\left (H_{\alpha_1}(z)z^{\alpha_1-1}\tilde{f}_1+aH(z)z^{\alpha_1-1}\tilde{f}_2\right )+\left (H_{\alpha_1}(z)z^{\alpha_1-1}G_{1,0}+aH(z)z^{\alpha_1-1}G_{2,0}\right ),\\
\tilde{G}_2=&\left (aH(z)z^{\alpha_2-1}\tilde{f}_1+H_{\alpha_2}(z)z^{\alpha_2-1}\tilde{f}_2\right )+\left (aH(z)z^{\alpha_2-1}G_{1,0}+H_{\alpha_2}(z)z^{\alpha_2-1}G_{2,0}\right ).
\end{aligned}
\end{equation}

To obtain the regularity of the solutions $G_1$ and $G_2$, we first provide some lemmas.
\begin{lemma}\label{lemestofa1easy}
	When $ z\in \Sigma_{\theta,\kappa}$ and $\kappa\geq2|a|^{1/\alpha_1}$, we have
	\begin{equation*}
		\left \|\left(z^{\alpha_1}+a+A\right)^{-1}\right \|\leq C|z|^{-\alpha_1}.
	\end{equation*}
\end{lemma}
\begin{proof}
	The fact $ z\in \Sigma_{\theta,\kappa}$ leads to that there exists $\pi/2<\bar{\theta}<\pi$,  $ z^{\alpha_1}+a\in \Sigma_{\bar{\theta}}$ holds, where $\Sigma_{\bar{\theta}}=\{z\in\mathbb{C}:|\arg z|\leq \bar{\theta}\}$. Combining the  resolvent estimate \cite{Lubich1996}
	\begin{equation*}
		\left \|\left(z+A\right)^{-1}\right \|\leq C|z|^{-1},\ z\in\Sigma_{\theta},
	\end{equation*}
we have
\begin{equation*}
\left \|\left(z^{\alpha_1}+a+A\right)^{-1}\right \|\leq C|z^{\alpha_1}+a|^{-1}.
\end{equation*}
Using the fact $|z|>2|a|^{1/\alpha_1}$ and
\begin{equation*}
\frac{|z^{\alpha_1}+a|^{-1}}{|z|^{-\alpha_1}}\leq 2,
\end{equation*}
 we get
\begin{equation*}
\left \|\left(z^{\alpha_1}+a+A\right)^{-1}\right \|\leq C|z|^{-\alpha_1}.
\end{equation*}
\end{proof}
Similarly, we can also get
\begin{lemma}\label{lemestofa2easy}
	When $ z\in \Sigma_{\theta,\kappa}$ and $\kappa\geq2|a|^{1/\alpha_2}$, there exists
	\begin{equation*}
		\left \|\left(z^{\alpha_2}+a+A\right)^{-1}\right \|\leq C|z|^{-\alpha_2}.
	\end{equation*}
\end{lemma}

\begin{lemma}\label{LemestofHz}
When $ z\in \Sigma_{\theta,\kappa}$ and $\kappa>\max\left (2|a|^{1/\alpha_1},2|a|^{1/\alpha_2}\right )$,
\begin{equation*}
\left \|H(z)\right\|\leq C|z|^{-\alpha_1-\alpha_2}
\end{equation*}
holds, where $H(z)$ is defined by \eqref{equdefHz}.
\end{lemma}
\begin{proof}
From \eqref{equdefHz}, let
\begin{equation*}
((z^{\alpha_1}+a+A)(z^{\alpha_2}+a+A)-a^2)u=v,
\end{equation*}
which leads to
\begin{equation}\label{qeuuu}
u=\left((z^{\alpha_1}+a+A)(z^{\alpha_2}+a+A)\right)^{-1}v+a^2\left((z^{\alpha_1}+a+A)(z^{\alpha_2}+a+A)\right)^{-1}u.
\end{equation}
Taking $L_2$ norm on both sides of (\ref{qeuuu}) and using Lemmas \ref{lemestofa1easy} and \ref{lemestofa2easy}, we have
\begin{equation*}
\|u\|_{L^2(\Omega)}\leq C|z|^{-\alpha_1-\alpha_2}\|v\|_{L^2(\Omega)}+Ca^2|z|^{-\alpha_1-\alpha_2}\|u\|_{L^2(\Omega)}.
\end{equation*}
When $\kappa$ is sufficiently large, such that $Ca^2|z|^{-\alpha_1-\alpha_2}<1/2$, then we have
\begin{equation*}
\|u\|_{L^2(\Omega)}\leq C|z|^{-\alpha_1-\alpha_2}\|v\|_{L^2(\Omega)},
\end{equation*}
which leads to the desired estimate.
\end{proof}

\begin{lemma}\label{LemestofHzalpha1}
When $ z\in \Sigma_{\theta,\kappa}$ and $\kappa>\max\left (2|a|^{1/\alpha_1},2|a|^{1/\alpha_2}\right )$,
\begin{equation*}
\left \|H_{\alpha_1}(z)\right\|\leq C|z|^{-\alpha_1}
\end{equation*}
holds, where $H_{\alpha_1}(z)$ is defined by \eqref{H12}.
\end{lemma}
\begin{proof}
By \eqref{H12}, let
\begin{equation*}
(z^{\alpha_1}+a+A)u-a^2\left(z^{\alpha_2}+a+A\right)^{-1}u=v,
\end{equation*}
which implies that
\begin{equation}\label{equddd}
u=a^2\left((z^{\alpha_1}+a+A)(z^{\alpha_2}+a+A)\right)^{-1}u+\left(z^{\alpha_1}+a+A\right)^{-1}v.
\end{equation}
Performing $L_2$ norm on both sides of (\ref{equddd}) and using Lemmas \ref{lemestofa1easy} and \ref{lemestofa2easy}, we have
\begin{equation*}
\|u\|_{L^2(\Omega)}\leq C|z|^{-\alpha_1}\|v\|_{L^2(\Omega)}+Ca^2|z|^{-\alpha_1-\alpha_2}\|u\|_{L^2(\Omega)}.
\end{equation*}
Taking $\kappa$ sufficiently large to ensure $Ca^2|z|^{-\alpha_1-\alpha_2}<1/2$, we have
\begin{equation*}
\|u\|_{L^2(\Omega)}\leq C|z|^{-\alpha_1}\|v\|_{L^2(\Omega)},
\end{equation*}
which leads to the desired estimate.
\end{proof}
Similarly, we can also obtain
\begin{lemma}\label{LemestofHzalpha2}
When $ z\in \Sigma_{\theta,\kappa}$, where $\kappa>\max\left (2|a|^{1/\alpha_1},2|a|^{1/\alpha_2}\right )$, there exists
\begin{equation*}
\left \|H_{\alpha_2}(z)\right\|\leq C|z|^{-\alpha_2},
\end{equation*}
 where $H_{\alpha_2}(z)$ is defined by \eqref{H12}.
\end{lemma}

\begin{lemma}\label{lemestofAHz}
	When $ z\in \Sigma_{\theta,\kappa}$, where $\kappa>\max\left (2|a|^{1/\alpha_1},2|a|^{1/\alpha_2}\right )$, one has
	\begin{equation*}
	\left \|AH(z)\right\|\leq C\min(|z|^{-\alpha_1},|z|^{-\alpha_2}).
	\end{equation*}
\end{lemma}
\begin{proof}
	First, there exist the equalities
	\begin{equation*}
		\begin{aligned}
		AH(z)&=(z^{\alpha_1}+a+A)H(z)-(z^{\alpha_1}+a)H(z)\\
		&=(z^{\alpha_2}+a+A)H(z)-(z^{\alpha_2}+a)H(z).
		\end{aligned}
	\end{equation*}
	To estimate $AH(z)$, we need to estimate $(z^{\alpha_1}+a+A)H(z)$, $(z^{\alpha_1}+a)H(z)$, $(z^{\alpha_2}+a+A)H(z)$, and $(z^{\alpha_2}+a)H(z)$. As for $(z^{\alpha_1}+a+A)H(z)$, let
	\begin{equation*}
		(z^{\alpha_2}+a+A)u-a^2(z^{\alpha_1}+a+A)^{-1}u=v,
	\end{equation*}
which results in
	\begin{equation*}
		u={\left(z^{\alpha_2}+a+A\right)^{-1}}v+a^2{\left((z^{\alpha_1}+a+A)(z^{\alpha_2}+a+A)\right)^{-1}}u.
	\end{equation*}
	Taking $\kappa$ sufficiently large leads to
	\begin{equation}
		\|(z^{\alpha_1}+a+A)H(z)\|\leq C|z|^{-\alpha_2}.
	\end{equation}
	Similarly, we have
	\begin{equation*}
		\begin{aligned}
		&\|(z^{\alpha_1}+a)H(z)\|\leq C|z|^{-\alpha_2},\\ &\|(z^{\alpha_2}+a)H(z)\|\leq C|z|^{-\alpha_1},\\
		&\|(z^{\alpha_2}+a+A)H(z)\|\leq C|z|^{-\alpha_1}.	
		\end{aligned}
	\end{equation*}
	Thus, the proof is completed.
\end{proof}
Similarly, we have the estimate
\begin{lemma}\label{LemestofAHzalpha}
	When $ z\in \Sigma_{\theta,\kappa}$ and $\kappa>\max\left (2|a|^{1/\alpha_1},2|a|^{1/\alpha_2}\right )$,
	\begin{equation*}
	\left \|AH_{\alpha_1}(z)\right\|\leq C
	\,\,\,{\rm and} \,\,\,
	\left \|AH_{\alpha_2}(z)\right\|\leq C
	\end{equation*}
hold.
\end{lemma}
Now, we are ready to provide the priori estimates for the solutions $G_1$ and $G_2$ of the homogeneous problem \eqref{equrqtosol} with both smooth and nonsmooth initial values.
\begin{theorem}\label{thmhomoregularity}
When $f_1=0$, $ f_2=0 $, assuming $G_{1,0},G_{2,0}\in L^2(\Omega)$, we have
\begin{equation}
\begin{aligned} 
&\|A^\nu G_1(t)\|_{L^2(\Omega)}\leq C\left (t^{-\nu\alpha_1}\|G_{1,0}\|_{L^2(\Omega)}+\|G_{2,0}\|_{L^2(\Omega)}\right ),\\
&\|A^\nu G_2(t)\|_{L^2(\Omega)}\leq C\left (\|G_{1,0}\|_{L^2(\Omega)}+t^{-\nu\alpha_2}\|G_{2,0}\|_{L^2(\Omega)}\right )
\end{aligned}
\end{equation}
for $ \nu=0,1 $; furthermore, if $G_{1,0},G_{2,0}\in H^1_0(\Omega)\bigcap H^2(\Omega)$, there exist
\begin{equation}
\begin{aligned}
&\|A G_1(t)\|_{L^2(\Omega)}\leq C\left (\|AG_{1,0}\|_{L^2(\Omega)}+\|AG_{2,0}\|_{L^2(\Omega)}\right ),\\
&\|A G_2(t)\|_{L^2(\Omega)}\leq C\left (\|AG_{1,0}\|_{L^2(\Omega)}+\|AG_{2,0}\|_{L^2(\Omega)}\right ).
\end{aligned}
\end{equation}
\end{theorem}
\begin{proof}
For given $t$, we take $\kappa\geq1/t$ and ensure that $\kappa$ is large enough to satisfy the conditions in Lemmas \ref{lemestofa1easy}--\ref{LemestofAHzalpha}. Letting $\mathcal{L}^{-1}$ denote the inverse Laplace transform and taking the inverse Laplace transform on \eqref{equsolformofLaplace}, we have
\begin{equation}\label{equregularinveseform1}
\begin{aligned}
G_1(t)=&\mathcal{L}^{-1}\{H_{\alpha_1}(z)z^{\alpha_1-1}G_{1,0}+aH(z)z^{\alpha_1-1}G_{2,0}\}
\end{aligned}
\end{equation}
and
\begin{equation}\label{equregularinveseform2}
\begin{aligned}
G_2(t)=&\mathcal{L}^{-1}\{aH(z)z^{\alpha_2-1}G_{1,0}+H_{\alpha_2}(z)z^{\alpha_2-1}G_{2,0}\}.
\end{aligned}
\end{equation}
Taking $L_2$ norm on both sides of (\ref{equregularinveseform1}) and using Lemmas \ref{LemestofHz} and \ref{LemestofHzalpha1}, there exists
\begin{equation}
\begin{aligned}
\|G_1(t)\|_{L^2(\Omega)}=&\left \|\mathcal{L}^{-1}\{H_{\alpha_1}(z)z^{\alpha_1-1}G_{1,0}+aH(z)z^{\alpha_1-1}G_{2,0}\}\right \|_{L^2(\Omega)}\\
\leq& C\int_{\Gamma_{\theta,\kappa}}e^{\Re(z)t}\left (|z|^{-1}\|G_{1,0}\|_{L^2(\Omega)}+|z|^{-\alpha_2-1}\|G_{2,0}\|_{L^2(\Omega)}\right )|dz|,
\end{aligned}
\end{equation}
where $\Re(z)$ denotes the real part of $z$. Since $1< T/t$, we have
\begin{equation*}
\begin{aligned}
&\|G_1(t)\|_{L^2(\Omega)}\\
\leq&C\left (\int^\infty_{\kappa}r^{-1}e^{rt\cos(\theta)}dr+\int^{\theta}_{-\theta} e^{\cos(\psi)\kappa t}d\psi\right )\|G_{1,0}\|_{L^2(\Omega)}\\
&+C\left (\int^\infty_{\kappa}r^{-\alpha_2-1}e^{rt\cos(\theta)}dr+\int^{\theta}_{-\theta}\kappa^{-\alpha_2}e^{\cos(\psi)\kappa t}d\psi\right )\|G_{2,0}\|_{L^2(\Omega)}\\
\leq & C(\|G_{1,0}\|_{L^2(\Omega)}+\|G_{2,0}\|_{L^2(\Omega)}).
\end{aligned}
\end{equation*}
Similarly, one can also get
\begin{equation*}
\|G_2(t)\|_{L^2(\Omega)}\leq C(\|G_{1,0}\|_{L^2(\Omega)}+\|G_{2,0}\|_{L^2(\Omega)}).
\end{equation*}
To get the bound of $ \|AG_1(t)\|_{L^2(\Omega)} $, let the operator $A$ act on both sides of \eqref{equregularinveseform1} and obtain
\begin{equation}\label{eqAG}
\begin{aligned}
AG_1(t)=&\mathcal{L}^{-1}\{A\left (H_{\alpha_1}(z)z^{\alpha_1-1}G_{1,0}+aH(z)z^{\alpha_1-1}G_{2,0}\right )\}.
\end{aligned}
\end{equation}
Taking $L_2$ norm on both sides of (\ref{eqAG}) and using Lemmas \ref{lemestofAHz} and \ref{LemestofAHzalpha}, we have
\begin{equation}\label{AGm}
\begin{aligned}
\|AG_1(t)\|_{L^2(\Omega)}=&\left \|\mathcal{L}^{-1}\{A\left (H_{\alpha_1}(z)z^{\alpha_1-1}G_{1,0}+aH(z)z^{\alpha_1-1}G_{2,0}\right )\}\right \|_{L^2(\Omega)}\\
\leq& C\int_{\Gamma_{\theta,\kappa}}e^{\Re(z)t}\left (|z|^{\alpha_1-1}\|G_{1,0}\|_{L^2(\Omega)}+|z|^{-1}\|G_{2,0}\|_{L^2(\Omega)}\right )|dz|,
\end{aligned}
\end{equation}
where the fact $ \min(|z|^{-\alpha_1},|z|^{-\alpha_2})\leq |z|^{-\alpha_1}$ is used. Since  $\kappa<\kappa T/t$, we obtain
\begin{equation*}
\begin{aligned}
&\|AG_1(t)\|_{L^2(\Omega)}\\
\leq&C\left (\int^\infty_{\kappa}r^{\alpha_1-1}e^{rt\cos(\theta)}dr+\int^{\theta}_{-\theta}\kappa^{\alpha_1}e^{\cos(\psi)\kappa t}d\psi\right )\|G_{1,0}\|_{L^2(\Omega)}\\
&+C\left (\int^\infty_{\kappa}r^{-1}e^{rt\cos(\theta)}dr+\int^{\theta}_{-\theta}e^{\cos(\psi)\kappa t}d\psi\right )\|G_{2,0}\|_{L^2(\Omega)}\\
\leq& C(t^{-\alpha_1}\|G_{1,0}\|_{L^2(\Omega)}+\|G_{2,0}\|_{L^2(\Omega)}).
\end{aligned}
\end{equation*}
Analogously, there exists
\begin{equation*}
\|AG_2(t)\|_{L^2(\Omega)}\leq C(\|G_{1,0}\|_{L^2(\Omega)}+t^{-\alpha_2}\|G_{2,0}\|_{L^2(\Omega)}).
\end{equation*}
At the same time, from \eqref{AGm} we can also obtain
\begin{equation*}
\|AG_1(t)\|_{L^2(\Omega)}\leq C(\|AG_{1,0}\|_{L^2(\Omega)}+\|AG_{2,0}\|_{L^2(\Omega)})
\end{equation*}
and
\begin{equation*}
\|AG_2(t)\|_{L^2(\Omega)}\leq C(\|AG_{1,0}\|_{L^2(\Omega)}+\|AG_{2,0}\|_{L^2(\Omega)}).
\end{equation*}
This completes the proof of this lemma.
\end{proof}
Next, we obtain the estimates of the solutions $G_1$ and $G_2$ for inhomogeneous problem \eqref{equrqtosol} with vanishing initial value.
\begin{theorem}\label{thminhomoregularity}
	When the initial value $\mathbf{G}_0=[0,0]^T$, we have
	\begin{equation}
	\begin{aligned}
	&\|A^\nu G_1(t)\|_{L^2(\Omega)}\leq C\left (\int_0^ts^{-\nu\alpha_1}\|f_1(t-s)\|_{L^2(\Omega)}ds+\int_{0}^t\|f_2(t-s)\|_{L^2(\Omega)}ds\right ),\\
	&\|A^\nu G_2(t)\|_{L^2(\Omega)}\leq C\left (\int_0^t\|f_1(t-s)\|_{L^2(\Omega)}ds+\int_0^ts^{-\nu\alpha_2}\|f_2(t-s)\|_{L^2(\Omega)}ds\right )
	\end{aligned}
	\end{equation}
for $\nu=0,1$.
\end{theorem}
\begin{proof}
For given $t$, we take $\kappa\geq1/t$ and ensure that $\kappa$ is large enough to satisfy the conditions in Lemmas \ref{lemestofa1easy}--\ref{LemestofAHzalpha}. Performing the inverse Laplace transform on \eqref{equsolformofLaplace} leads to
	\begin{equation}\label{equregularinveseformF1}
	\begin{aligned}
	G_1(t)=&\mathcal{L}^{-1}\left \{H_{\alpha_1}(z)z^{\alpha_1-1}\tilde{f}_1+aH(z)z^{\alpha_1-1}\tilde{f}_2\right \}
	\end{aligned}
	\end{equation}
	and
	\begin{equation}\label{equregularinveseformF2}
	\begin{aligned}
	G_2(t)=&\mathcal{L}^{-1}\left \{aH(z)z^{\alpha_2-1}\tilde{f}_1+H_{\alpha_2}(z)z^{\alpha_2-1}\tilde{f}_2\right \}.
	\end{aligned}
	\end{equation}
	Taking $L_2$ norm on both sides of (\ref{equregularinveseformF1}) and (\ref{equregularinveseformF2}), and using Lemmas \ref{LemestofHz} and \ref{LemestofHzalpha1}, we have
	\begin{equation}
	\begin{aligned}
&	\|G_1(t)\|_{L^2(\Omega)}
\\
&=\left \|\mathcal{L}^{-1}\{H_{\alpha_1}(z)z^{\alpha_1-1}\tilde{f}_1+aH(z)z^{\alpha_1-1}\tilde{f}_2\}\right \|_{L^2(\Omega)}\\
	& \leq C\int_{\Gamma_{\theta,\kappa}}e^{\Re(z)t}|z|^{-1}dz\ast\|f_1\|_{L^2(\Omega)}+C\int_{\Gamma_{\theta,\kappa}}e^{\Re(z)t}|z|^{-1-\alpha_2}dz\ast\|f_2\|_{L^2(\Omega)},
	\end{aligned}
	\end{equation}
	where `$\ast$' denotes the convolution and the convolution rule of the inverse Laplace transform  $\mathcal{L}^{-1}\{\widetilde{uv}\}=\mathcal{L}^{-1}\{\tilde{u}\}\ast v$ is used.
	Being similar to the proof of Theorem \ref{thmhomoregularity}, one can get
	\begin{equation*}
	\|G_1(t)\|_{L^2(\Omega)}\leq C\left (1\ast\|f_1\|_{L^2(\Omega)}+1\ast\|f_2\|_{L^2(\Omega)}\right ).
	\end{equation*}
And also there exists
	\begin{equation*}
	\|G_2(t)\|_{L^2(\Omega)}\leq C\left (1\ast\|f_1\|_{L^2(\Omega)}+1\ast\|f_2\|_{L^2(\Omega)}\right ).
	\end{equation*}
	To estimate $ \|AG_1(t)\|_{L^2(\Omega)} $, we have
	\begin{equation}\label{eqAG1111}
	\begin{aligned}
	AG_1(t)=&\mathcal{L}^{-1}\left \{A\left (H_{\alpha_1}(z)z^{\alpha_1-1}\tilde{f}_1+aH(z)z^{\alpha_1-1}\tilde{f}_2\right )\right \}.
	\end{aligned}
	\end{equation}
	Taking $L_2$ norm on both sides of (\ref{eqAG1111}) and using Lemmas \ref{lemestofAHz} and \ref{LemestofAHzalpha}, we have
		\begin{equation*}
	\begin{aligned}
	\|AG_1(t)\|_{L^2(\Omega)}=&\left \|\mathcal{L}^{-1}\left \{A\left (H_{\alpha_1}(z)z^{\alpha_1-1}\tilde{f}_1+aH(z)z^{\alpha_1-1}\tilde{f}_2\right )\right \}\right \|_{L^2(\Omega)}\\
	\leq& C\int_{\Gamma_{\theta,\kappa}}e^{\Re(z)t}(|z|^{\alpha_1-1}\|\tilde{f}_1\|_{L^2(\Omega)}+|z|^{-1}\|\tilde{f}_2\|_{L^2(\Omega)})|dz|,
	\end{aligned}
	\end{equation*}
	where the fact $ \min(|z|^{-\alpha_1},|z|^{-\alpha_2})\leq |z|^{-\alpha_1}$ is used. Thus we have
	\begin{equation*}
	\|AG_1(t)\|_{L^2(\Omega)}\leq C\left (t^{-\alpha_1}\ast\|f_1(t)\|_{L^2(\Omega)}+1\ast\|f_2(t)\|_{L^2(\Omega)}\right ).
	\end{equation*}
The estimate of $\|AG_2(t)\|_{L^2(\Omega)}$ can be, similarly, obtained.
\end{proof}

\section{Space discretization and error analysis}
In this section, we discretize the space derivatives by the finite element method and provide the error estimates for the space semidiscrete scheme of the homogeneous problem \eqref{equrqtosol} with smooth and nonsmooth initial values and the inhomogeneous problem \eqref{equrqtosol} with vanishing initial value. Let $\mathcal{T}_h$ be a shape regular quasi-uniform partitions of the domain $\Omega$, where $h$ is the maximum diameter. Denote $ X_h $ as piecewise linear finite element space
\begin{equation*}
	X_h=\{v_h\in H^1_0(\Omega): v_h|_\mathbf{T}\ {\rm is\ a \ linear\ function} \ \forall \mathbf{T}\in\mathcal{T}_h\}.
\end{equation*}

Then we define the $ L^2 $-orthogonal projection $ P_h:\ L^2(\Omega)\rightarrow X_h $ and the Ritz projection $ R_h:\ H^1_0(\Omega)\rightarrow X_h $ \cite{Bazhlekova2015}, respectively, by
\begin{equation*}
	\begin{aligned}
	 &(P_hu,v_h)=(u,v_h) \ ~\forall v_h\in X_h,\\
	 &(\nabla R_hu,\nabla v_h)=(\nabla u,\nabla v_h) \ ~\forall v_h\in X_h.
	\end{aligned}
\end{equation*}
And denote
\begin{equation}\label{Hh}
H_h(z)=\left((z^{\alpha_1}+a+A_h)(z^{\alpha_2}+a+A_h)-a^2\right)^{-1},
\end{equation}
and
\begin{equation}\label{Halpha}
\begin{aligned}
H_{\alpha_1,h}=H_h(z)(z^{\alpha_2}+a+A_h),\quad
H_{\alpha_2,h}=H_h(z)(z^{\alpha_1}+a+A_h).
\end{aligned}
\end{equation}

The $ L^2 $-orthogonal projection $ P_h $ and  the Ritz projection $ R_h $ have the following approximation properties.
\begin{lemma}[\cite{Bazhlekova2015}]\label{lemprojection}
	The projection $ P_h $ and $ R_h $ satisfy
	\begin{equation*}
	\begin{aligned}
		&\|P_hu-u\|_{L^2(\Omega)}+h\|\nabla(P_hu-u)\|_{L^2(\Omega)}\leq Ch^q\|u\|_{\dot{H}^{q}(\Omega)}\ for\ u\in \dot{H}^q(\Omega),\ q=1,2,\\
		&\|R_hu-u\|_{L^2(\Omega)}+h\|\nabla(R_hu-u)\|_{L^2(\Omega)}\leq Ch^q\|u\|_{\dot{H}^{q}(\Omega)}\ for\ u\in \dot{H}^q(\Omega),\ q=1,2.
	\end{aligned}
	\end{equation*}
\end{lemma}

Denote $(\cdot,\cdot)$ as the $L_2$ inner product. The semidiscrete Galerkin scheme for system \eqref{equrqtosol} reads: Find $ G_{1,h},G_{2,h}\in X_h$ such that
\begin{equation}\label{equequsysspatialsemischeme}
	\left\{\begin{aligned}
	&\left (\frac{\partial G_{1,h}}{\partial t},v_{1,h}\right )+a~_0D^{1-\alpha_1}_t(G_{1,h},v_{1,h})\\
    &\quad\quad+~_0D^{1-\alpha_1}_t(\nabla G_{1,h},\nabla v_{1,h})=a~_0D^{1-\alpha_2}_t(G_{2,h},v_{1,h})+(f_1,v_{1,h}),\\
	&\left (\frac{\partial G_{2,h}}{\partial t},v_{2,h}\right )+a~_0D^{1-\alpha_2}_t(G_{2,h},v_{2,h})\\
    &\quad\quad+~_0D^{1-\alpha_2}_t(\nabla G_{2,h},\nabla v_{2,h})=a~_0D^{1-\alpha_1}_t(G_{1,h},v_{2,h})+(f_2,v_{2,h}),\\
	\end{aligned}\right.
\end{equation}
where $v_{1,h},\ v_{2,h}\in X_h$. As for $G_{1,h}(0)$ and $ G_{2,h}(0) $, we take $G_{1,h}(0)=P_hG_{1,0}$, $G_{2,h}(0)=P_hG_{2,0}$ if $G_{1,h}, G_{2,h}\in L^2(\Omega)$ and $G_{1,h}(0)=R_hG_{1,0}$, $G_{2,h}(0)=R_hG_{2,0}$ if $G_{1,h}, G_{2,h}\in H^1_0(\Omega)\bigcap H^2(\Omega)$.

Define the discrete operator $A_h$: $X_h\rightarrow X_h$,
\begin{equation*}
	(A_h u_h,v_h)=(\nabla u_h,\nabla v_h) \quad\forall u_h,\ v_h\in X_h
\end{equation*}
and $ f_{1,h}(t)=P_hf_1(t)$, $ f_{2,h}(t)=P_hf_2(t)$; then  \eqref{equequsysspatialsemischeme}  can be rewritten as
\begin{equation}\label{equequsysspatialsemiAh}
	\begin{aligned}
		&\frac{\partial G_{1,h}}{\partial t}+a~_0D^{1-\alpha_1}_tG_{1,h}+~_0D^{1-\alpha_1}_tA_h G_{1,h}=a~_0D^{1-\alpha_2}_tG_{2,h}+f_{1,h},\\
		&\frac{\partial G_{2,h}}{\partial t}+a~_0D^{1-\alpha_2}_tG_{2,h}+~_0D^{1-\alpha_2}_tA_h G_{2,h}=a~_0D^{1-\alpha_1}_tG_{1,h}+f_{2,h}.\\
	\end{aligned}
\end{equation}
Taking the Laplace transform on \eqref{equequsysspatialsemiAh}, we get
\begin{equation}\label{equequsysspatialsemilapform}
\begin{aligned}
	&z\tilde{G}_{1,h}+az^{1-\alpha_1}\tilde{G}_{1,h}+z^{1-\alpha_1}A_h\tilde{G}_{1,h}=az^{1-\alpha_2}\tilde{G}_{2,h}+\tilde{f}_{1,h}+G_{1,h}(0),\\
	&z\tilde{G}_{2,h}+az^{1-\alpha_2}\tilde{G}_{2,h}+z^{1-\alpha_2}A_h\tilde{G}_{2,h}=az^{1-\alpha_1}\tilde{G}_{1,h}+\tilde{f}_{2,h}+G_{2,h}(0).
	\end{aligned}
\end{equation}

Then we introduce two lemmas, which will be used in the error estimates for space semidiscrete scheme.
\begin{lemma}\label{lemestofAhsoeasy}
When $ z\in \Sigma_{\theta,\kappa}$ and $\kappa=2|a|^{1/\alpha_1}$, we have
	\begin{equation*}
    \begin{aligned}
		\left \|\left(z^{\alpha_1}+a+A_h\right)^{-1}\right \|\leq C|z|^{-\alpha_1},\\
        \left \|\left(z^{\alpha_2}+a+A_h\right)^{-1}\right \|\leq C|z|^{-\alpha_2}.
    \end{aligned}
	\end{equation*}
\end{lemma}
\begin{proof}
Its proof is similar to the one of Lemma \ref{lemestofa1easy}.
\end{proof}

Being similar to $H(z)$, $H_{\alpha_1}(z)$, and $H_{\alpha_2}(z)$, we have the following estimates of $H_h(z)$, $H_{\alpha_1,h}(z)$, and $H_{\alpha_2,h}(z)$.
\begin{lemma}\label{lemtheestimateofHhblabla}
	When $z\in\Sigma_{\theta,\kappa}$, $ \pi/2<\theta<\pi $, and $\kappa>\max\left (2|a|^{1/\alpha_1},2|a|^{1/\alpha_2}\right )$, there are the estimates of $H_h(z)$, $H_{\alpha_1,h}(z)$ and $H_{\alpha_2,h}(z)$,
	\begin{equation*}
	\begin{aligned}
	&\|H_h(z)\|\leq C|z|^{-\alpha_1-\alpha_2},\quad\|H_{\alpha_1,h}(z)\|\leq C|z|^{-\alpha_1},\quad\|H_{\alpha_2,h}(z)\|\leq C|z|^{-\alpha_2},\\
	&\|A_hH_h(z)\|\leq C\min\left (|z|^{-\alpha_1},|z|^{-\alpha_2}\right ),\quad\|A_hH_{\alpha_1,h}(z)\|\leq C,\quad\|A_hH_{\alpha_2,h}(z)\|\leq C,\\
	\end{aligned}
	\end{equation*}
	where $H_h$, $H_{\alpha_1,h}$ and $H_{\alpha_2,h}$ are defined by \eqref{Hh} and \eqref{Halpha}.
\end{lemma}

 For homogeneous problem \eqref{equrqtosol}, we give the error estimates for space semidiscrete scheme with smooth initial value.
\begin{theorem}\label{thmsmoothdatasemi}
	Let $G_{1}$, $G_{2}$ and $G_{1,h}$, $G_{2,h}$ be the solutions of the systems \eqref{equrqtosol} and \eqref{equequsysspatialsemiAh}, respectively, with $f_1=0$, $f_2=0$ and $G_{1,h}(0)=R_hG_{1,0}$, $G_{2,h}(0)=R_hG_{2,0}$. Then
	\begin{equation*}
	\begin{aligned}
	&\|G_1-G_{1,h}\|_{L^2(\Omega)}\leq Ch^2\left (\|G_{1,0}\|_{\dot{H}^2(\Omega)}+\|G_{2,0}\|_{\dot{H}^2(\Omega)}\right ),\\
	&\|G_2-G_{2,h}\|_{L^2(\Omega)}\leq Ch^2\left (\|G_{1,0}\|_{\dot{H}^2(\Omega)}+\|G_{2,0}\|_{\dot{H}^2(\Omega)}\right ).
	\end{aligned}
	\end{equation*}
\end{theorem}
\begin{proof}
	For given $t$, we take $\kappa\geq1/t$ and ensure that $\kappa$ is large enough to satisfy the conditions in Lemmas \ref{lemestofa1easy}--\ref{LemestofAHzalpha} and Lemma \ref{lemtheestimateofHhblabla}. For $ G_1 $ and $G_2$, we have
	\begin{equation*}
	\begin{aligned}
	\begin{aligned}
	G_1-G_{1,h}=&(R_hG_1-G_{1,h})+(G_1-R_hG_1)=\varrho_1+\varUpsilon_1,\\
	G_2-G_{2,h}=&(R_hG_2-G_{2,h})+(G_2-R_hG_2)=\varrho_2+\varUpsilon_2.
	\end{aligned}
	\end{aligned}
	\end{equation*}
	Lemma \ref{lemprojection} and Theorem \ref{thmhomoregularity} lead to
	\begin{equation}\label{equRhGG}
	\begin{aligned}  
	&\|\varUpsilon_1\|_{L^2(\Omega)}\leq Ch^2\|G_1(t)\|_{\dot{H}^2(\Omega)}\leq Ch^2(\|G_{1,0}\|_{\dot{H}^2(\Omega)}+\|G_{2,0}\|_{\dot{H}^2(\Omega)}),\\
	&\|\varUpsilon_2\|_{L^2(\Omega)}\leq Ch^2\|G_2(t)\|_{\dot{H}^2(\Omega)}\leq Ch^2(\|G_{1,0}\|_{\dot{H}^2(\Omega)}+\|G_{2,0}\|_{\dot{H}^2(\Omega)}).
	\end{aligned}
	\end{equation}
	Applying the operator $P_h$ on both sides of the first formula in \eqref{equequinlapform}, we have
	\begin{equation}\label{equPhform1}
\begin{aligned}	& zP_h\tilde{G}_1+az^{1-\alpha_1}P_h\tilde{G}_1+z^{1-\alpha_1}P_hA\tilde{G}_1=az^{1-\alpha_2}P_h\tilde{G}_2+P_hG_{1,0},\\
& 	zP_h\tilde{G}_2+az^{1-\alpha_2}P_h\tilde{G}_2+z^{1-\alpha_2}P_hA\tilde{G}_2=az^{1-\alpha_1}P_h\tilde{G}_1+P_hG_{2,0}.	
\end{aligned}
\end{equation}
	Subtracting \eqref{equPhform1} from  \eqref{equequsysspatialsemilapform} and using the fact $ A_hR_h=P_hA $, we have
	\begin{equation*}
	\begin{aligned}
	&z\tilde{\varrho}_1+az^{1-\alpha_1}\tilde{\varrho}_1+z^{1-\alpha_1}A_h\tilde{\varrho}_1\\
	&\qquad\qquad=az^{1-\alpha_2}\tilde{\varrho}_2+(R_h-P_h)(z\tilde{G}_1+az^{1-\alpha_1}\tilde{G}_1-az^{1-\alpha_2}\tilde{G}_2-G_{1,0}),\\
	&z\tilde{\varrho}_2+az^{1-\alpha_2}\tilde{\varrho}_2+z^{1-\alpha_2}A_h\tilde{\varrho}_2\\
	&\qquad\qquad =az^{1-\alpha_1}\tilde{\varrho}_1+(R_h-P_h)(z\tilde{G}_2+az^{1-\alpha_2}\tilde{G}_2-az^{1-\alpha_1}\tilde{G}_1-G_{2,0}).
	\end{aligned}
	\end{equation*}
	Thus we have
	\begin{equation*}
	\begin{aligned}
	\tilde{\varrho}_1=&z^{\alpha_1-1}\left(H_{\alpha_1,h}(z)(R_h-P_h)(z\tilde{G}_1+az^{1-\alpha_1}\tilde{G}_1-az^{1-\alpha_2}\tilde{G}_2-G_{1,0})\right.\\
	&\left.+aH_h(z)(R_h-P_h)(z\tilde{G}_2+az^{1-\alpha_2}\tilde{G}_2-az^{1-\alpha_1}\tilde{G}_1-G_{2,0})\right).
	\end{aligned}
	\end{equation*}
	By \eqref{equsolformofLaplace}, we have
	\begin{equation}\label{eqp1}
	\begin{aligned}
	\tilde{\varrho}_1=&H_{\alpha_1,h}(z)(R_h-P_h)\left(H_{\alpha_1}(z)z^{2\alpha_1-1}G_{1,0}+aH(z)z^{2\alpha_1-1}G_{2,0}\right.\\
	&+a H_{\alpha_1}(z)z^{\alpha_1-1}G_{1,0}+a^2H(z)z^{\alpha_1-1}G_{2,0}\\
	&\left.-(a^2H(z)z^{\alpha_1-1}G_{1,0}+aH_{\alpha_2}(z)z^{\alpha_1-1}G_{2,0})-z^{\alpha_1-1}G_{1,0}\right)\\
	&+aH_h(z)(R_h-P_h)\left(aH(z)z^{\alpha_1+\alpha_2-1}G_{1,0}+H_{\alpha_2}(z)z^{\alpha_1+\alpha_2-1}G_{2,0}\right.\\
	&+a^2H(z)z^{\alpha_1-1}G_{1,0}+aH_{\alpha_2}(z)z^{\alpha_1-1}G_{2,0}\\
	&\left.-(aH_{\alpha_1}(z)z^{\alpha_1-1}G_{1,0}+a^2H(z)z^{\alpha_1-1}G_{2,0})-z^{\alpha_1-1}G_{2,0}\right).
	\end{aligned}
	\end{equation}
	Taking the inverse Laplace transform and $L_2$ norm on both sides of \eqref{eqp1}, we obtain
	\begin{equation*}
	\begin{aligned}
	&\|\varrho_1\|_{L^2(\Omega)}\\
	\leq& \Big \|\frac{1}{2\pi i}\int_{\Gamma_{\theta,\kappa}}e^{zt}H_{\alpha_1,h}(z)(R_h-P_h)(H_{\alpha_1}(z)z^{2\alpha_1-1}+aH_{\alpha_1}(z)z^{\alpha_1-1}-a^2H(z)z^{\alpha_1-1}
\\
&
-z^{\alpha_1-1})G_{1,0}\Big \|_{L^2(\Omega)}+\Big \|\frac{1}{2\pi i}\int_{\Gamma_{\theta,\kappa}}e^{zt}H_{\alpha_1,h}(z)(R_h-P_h)(a H(z)z^{2\alpha_1-1}
\\
&
+a^2 H(z)z^{\alpha_1-1}-aH_{\alpha_2}(z)z^{\alpha_1-1})G_{2,0}\Big \|_{L^2(\Omega)}+\Big \|\frac{1}{2\pi i}\int_{\Gamma_{\theta,\kappa}}e^{zt}aH_{h}(z)(R_h-P_h)
\\
&\cdot(aH(z)z^{\alpha_1+\alpha_2-1}+a^2H(z)z^{\alpha_1-1}-aH_{\alpha_1}(z)z^{\alpha_1-1})G_{1,0}\Big \|_{L^2(\Omega)}\\
	&+\Big \|\frac{1}{2\pi i}\int_{\Gamma_{\theta,\kappa}}e^{zt}aH_{h}(z)(R_h-P_h)(H_{\alpha_2}(z)z^{\alpha_1+\alpha_2-1}+H_{\alpha_2}(z)az^{\alpha_1-1}
\\
	&
-a^2H(z)z^{\alpha_1-1}-z^{\alpha_1-1})G_{2,0}\Big \|_{L^2(\Omega)}.\\
	\end{aligned}
	\end{equation*}
	By Lemmas \ref{LemestofHz}, \ref{LemestofHzalpha1}, \ref{LemestofHzalpha2},  \ref{lemestofAhsoeasy}, and \ref{lemtheestimateofHhblabla}, there exists
	\begin{equation}\label{equRhG1G1h}
	\begin{aligned}
	&\|\varrho_1\|_{L^2(\Omega)}\\\leq
	&Ch^2\int_{\Gamma_{\theta,\kappa}}e^{\Re(z)t}(|z|^{-1}+|z|^{-1-\alpha_1}+|z|^{-\alpha_1-\alpha_2-1}+|z|^{-1})|dz|\|G_{1,0}\|_{\dot{H}^2(\Omega)}\\
	&+Ch^2\int_{\Gamma_{\theta,\kappa}}e^{\Re(z)t}(|z|^{-\alpha_2-1}+|z|^{-\alpha_1-\alpha_2-1}+|z|^{-\alpha_2-1})|dz|\|G_{2,0}\|_{\dot{H}^2(\Omega)}\\
	&+Ch^2\int_{\Gamma_{\theta,\kappa}}e^{\Re(z)t}(|z|^{-\alpha_1-\alpha_2-1}+|z|^{-\alpha_1-2\alpha_2-1}+|z|^{-\alpha_1-\alpha_2-1})|dz|\|G_{1,0}\|_{\dot{H}^2(\Omega)}\\	&+Ch^2\int_{\Gamma_{\theta,\kappa}}e^{\Re(z)t}(|z|^{-\alpha_2-1}+|z|^{-2\alpha_2-1}+|z|^{-\alpha_1-2\alpha_2-1}+|z|^{-\alpha_2-1})|dz|
\\
&
\cdot\|G_{2,0}\|_{\dot{H}^2(\Omega)}\\
	&\leq Ch^2\left (\|G_{1,0}\|_{\dot{H}^2(\Omega)}+\|G_{2,0}\|_{\dot{H}^2(\Omega)}\right),
	\end{aligned}
	\end{equation}
	where we use the fact $t\leq T$, and $\Re(z)$ stands for the real part of $z$.
	Similarly we have
	\begin{equation}\label{equRhG2G2h}
	\|\varrho_2\|_{L^2(\Omega)}\leq Ch^2\left (\|G_{1,0}\|_{\dot{H}^2(\Omega)}+\|G_{2,0}\|_{\dot{H}^2(\Omega)}\right).
	\end{equation}
	According to \eqref{equRhGG}, we get
	\begin{equation*}
	\begin{aligned}
	&\|G_1-G_{1,h}\|_{L^2(\Omega)}\leq Ch^2\left (\|G_{1,0}\|_{\dot{H}^2(\Omega)}+\|G_{2,0}\|_{\dot{H}^2(\Omega)}\right ),\\
	&\|G_2-G_{2,h}\|_{L^2(\Omega)}\leq Ch^2\left (\|G_{1,0}\|_{\dot{H}^2(\Omega)}+\|G_{2,0}\|_{\dot{H}^2(\Omega)}\right ).
	\end{aligned}
	\end{equation*}
\end{proof}

For homogeneous problem \eqref{equrqtosol} with nonsmooth initial value, we have the following error estimate.
\begin{theorem}\label{thmnonsmoothdatasemi}
	Let $G_{1}$, $G_{2}$ and $G_{1,h}$, $G_{2,h}$ be the solutions of the systems \eqref{equrqtosol} and \eqref{equequsysspatialsemiAh}, respectively, with $f_1=0$, $f_2=0$ and $G_{1,h}(0)=P_hG_{1,0}$, $G_{2,h}(0)=P_hG_{2,0}$. Then
	\begin{equation*}
	\begin{aligned}
	\|G_1-G_{1,h}\|_{L^2(\Omega)}\leq Ch^2(t^{-\alpha_1}\|G_{1,0}\|_{L^2(\Omega)}+\|G_{2,0}\|_{L^2(\Omega)}),\\
	\|G_2-G_{2,h}\|_{L^2(\Omega)}\leq Ch^2(\|G_{1,0}\|_{L^2(\Omega)}+t^{-\alpha_2}\|G_{2,0}\|_{L^2(\Omega)}).
	\end{aligned}
	\end{equation*}
\end{theorem}
\begin{proof}
	For given $t$, we take $\kappa\geq1/t$ and ensure that $\kappa$ is large enough to satisfy the conditions in Lemmas \ref{lemestofa1easy}--\ref{LemestofAHzalpha} and Lemma \ref{lemtheestimateofHhblabla}. For $ G_1 $ and $G_2$, there are
	\begin{equation*}
	\begin{aligned}
	\begin{aligned}
	G_1-G_{1,h}=&(P_hG_1-G_{1,h})+(G_1-P_hG_1)=\varrho_1+\varUpsilon_1,\\
	G_2-G_{2,h}=&(P_hG_2-G_{2,h})+(G_2-P_hG_2)=\varrho_2+\varUpsilon_2.
	\end{aligned}
	\end{aligned}
	\end{equation*}
	Lemma \ref{lemprojection} and Theorem \ref{thmhomoregularity} lead to
	\begin{equation}\label{equPhGG}
	\begin{aligned} 
	&\|\varUpsilon_1\|_{L^2(\Omega)}\leq Ch^2\|G_1(t)\|_{\dot{H}^2(\Omega)}\leq Ch^2(t^{-\alpha_1}\|G_{1,0}\|_{L^2(\Omega)}+\|G_{2,0}\|_{L^2(\Omega)}),\\
	&\|\varUpsilon_2\|_{L^2(\Omega)}\leq Ch^2\|G_2(t)\|_{\dot{H}^2(\Omega)}\leq Ch^2(\|G_{1,0}\|_{L^2(\Omega)}+t^{-\alpha_2}\|G_{2,0}\|_{L^2(\Omega)}).
	\end{aligned}
	\end{equation}
	Subtracting \eqref{equPhform1} from \eqref{equequsysspatialsemilapform} and using the fact $ A_hR_h=P_hA $, we have
	\begin{equation*}
	\begin{aligned}
	&z\tilde{\varrho}_1+az^{1-\alpha_1}\tilde{\varrho}_1+z^{1-\alpha_1}A_h\tilde{\varrho}_1=az^{1-\alpha_2}\tilde{\varrho}_2-z^{1-\alpha_1}A_h(R_h-P_h)\tilde{G}_1,\\
	&z\tilde{\varrho}_2+az^{1-\alpha_2}\tilde{\varrho}_2+z^{1-\alpha_2}A_h\tilde{\varrho}_2=az^{1-\alpha_1}\tilde{\varrho}_1-z^{1-\alpha_2}A_h(R_h-P_h)\tilde{G}_2,
	\end{aligned}
	\end{equation*}
which leads to
	\begin{equation*}
	\tilde{\varrho}_1=z^{\alpha_1-1}\left(H_{\alpha_1,h}(z)A_h(P_h-R_h)z^{1-\alpha_1}\tilde{G}_1+aH_h(z)A_h(P_h-R_h)z^{1-\alpha_2}\tilde{G}_2\right).
	\end{equation*}
Further combining \eqref{equsolformofLaplace} results in
	\begin{equation*}
	\begin{aligned}
	\tilde{\varrho}_1=&H_{\alpha_1,h}A_h(P_h-R_h)\left (H_{\alpha_1}(z)z^{\alpha_1-1}G_{1,0}+aH(z)z^{\alpha_1-1}G_{2,0}\right )\\
	&+az^{\alpha_1-\alpha_2}H_h(z)A_h(P_h-R_h)\left (aH(z)z^{\alpha_2-1}G_{1,0}+H_{\alpha_2}(z)z^{\alpha_2-1}G_{2,0}\right ).
	\end{aligned}
	\end{equation*}
	Taking the inverse Laplace transform and using Lemmas \ref{lemprojection},  \ref{lemestofAhsoeasy}, and \ref{lemtheestimateofHhblabla}, we have
	\begin{equation*}
	\begin{aligned}
	&\|\varrho_1\|_{L^2(\Omega)}\\\leq&\Big \|\frac{1}{2\pi i}\int_{\Gamma_{\theta,\kappa}}e^{zt}H_{\alpha_1,h}A_h(P_h-R_h)\left (H_{\alpha_1}(z)z^{\alpha_1-1}G_{1,0}+aH(z)z^{\alpha_1-1}G_{2,0}\right )dz\Big \|_{L^2(\Omega)}\\
	&+\Big \|\frac{1}{2\pi i}\int_{\Gamma_{\theta,\kappa}}e^{zt}az^{\alpha_1-\alpha_2}H_h(z)A_h(P_h-R_h)\left (aH(z)z^{\alpha_2-1}G_{1,0}+H_{\alpha_2}(z) \right.
\\
&
\cdot \left. z^{\alpha_2-1}G_{2,0}\right )dz\Big \|_{L^2(\Omega)}\\
	\leq&Ch^2\int_{\Gamma_{\theta,\kappa}}e^{\Re(z)t}\left (\|AH_{\alpha_1}(z)z^{\alpha_1-1}G_{1,0}\|_{L^2(\Omega)}+\|aAH(z)z^{\alpha_1-1}G_{2,0}\|_{L^2(\Omega)}\right )|dz|\\
	&+Ch^2\int_{\Gamma_{\theta,\kappa}}e^{\Re(z)t}|z|^{-\alpha_2}\left (\|AH(z)z^{\alpha_2-1}G_{1,0}\|_{L^2(\Omega)} \right.
\\
	&
\left.+\|AH_{\alpha_2}(z)z^{\alpha_2-1}G_{2,0}\|_{L^2(\Omega)}\right )|dz|.
	\end{aligned}
	\end{equation*}
	Combining Lemmas \ref{lemestofAHz} and \ref{LemestofAHzalpha} lead to
	\begin{equation*}
	\begin{aligned}
	\|\varrho_1\|_{L^2(\Omega)}\leq&Ch^2\int_{\Gamma_{\theta,\kappa}}e^{\Re(z)t}\left (|z|^{\alpha_1-1}\|G_{1,0}\|_{L^2(\Omega)}+|z|^{-1}\|G_{2,0}\|_{L^2(\Omega)}\right )|dz|\\
	&+Ch^2\int_{\Gamma_{\theta,\kappa}}e^{\Re(z)t}|z|^{-\alpha_2}\left (|z|^{-1}\|G_{1,0}\|_{L^2(\Omega)}+|z|^{\alpha_2-1}\|G_{2,0}\|_{L^2(\Omega)}\right )|dz|\\
	\leq&Ch^2\left (t^{-\alpha_1}\|G_{1,0}\|_{L^2(\Omega)}+\|G_{2,0}\|_{L^2(\Omega)}\right ).
	\end{aligned}
	\end{equation*}
	Similarly, there also exists
	\begin{equation*}
	\|\varrho_2\|_{L^2(\Omega)}\leq Ch^2(\|G_{1,0}\|_{L^2(\Omega)}+t^{-\alpha_2}\|G_{2,0}\|_{L^2(\Omega)})	.
	\end{equation*}
	According to \eqref{equPhGG}, we obtain
	\begin{equation*}
	\begin{aligned}
	\|G_1-G_{1,h}\|_{L^2(\Omega)}\leq Ch^2(t^{-\alpha_1}\|G_{1,0}\|_{L^2(\Omega)}+\|G_{2,0}\|_{L^2(\Omega)}),\\
	\|G_2-G_{2,h}\|_{L^2(\Omega)}\leq Ch^2(\|G_{1,0}\|_{L^2(\Omega)}+t^{-\alpha_2}\|G_{2,0}\|_{L^2(\Omega)}).
	\end{aligned}
	\end{equation*}
\end{proof}
Lastly, we provide error estimate of space semidiscrete scheme for inhomogeneous problem \eqref{equrqtosol} with vanishing initial value.
\begin{theorem}\label{thmvanisheddatasemi}
	Let $G_{1}$, $G_{2}$ and $G_{1,h}$, $G_{2,h}$ be the solutions of the systems \eqref{equrqtosol} and \eqref{equequsysspatialsemiAh}, respectively, with $G_{1,h}(0)=0$, $G_{2,h}(0)=0$ and $ f_1,f_2\in L^{\infty}(0,T,{L^2(\Omega)}) $. Then
	\begin{equation*}
	\begin{aligned}
	\|G_1-G_{1,h}\|_{L^2(\Omega)}\leq Ch^2\left (\int_0^t (t-s)^{-\alpha_1}\|f_1(s)\|_{L^2(\Omega)}ds+\int_0^t\|f_2(s)\|_{L^2(\Omega)}ds\right ),\\
	\|G_2-G_{2,h}\|_{L^2(\Omega)}\leq Ch^2\left (\int_0^t\|f_1(s)\|_{L^2(\Omega)}ds+\int_0^t(t-s)^{-\alpha_2}\|f_2(s)\|_{L^2(\Omega)}ds\right ).
	\end{aligned}
	\end{equation*}
\end{theorem}

\begin{proof}
	Its proof is similar to the one of Theorem \ref{thmnonsmoothdatasemi}.
\end{proof}

\section{Time discretization and error analysis}
In this section, we use the convolution quadrature to discretize the time fractional derivatives and perform the error analysis for the fully discrete scheme, in which the backward Euler method is used to get the first-order scheme for classical time derivative.
First, let the time step size $\tau=T/L$, $L\in\mathbb{N}$, $t_i=i\tau$, $i=0,1,\ldots,L$ and $0=t_0<t_1<\cdots<t_L=T$. Taking $\delta(\zeta)=(1-\zeta)$ and using convolution quadrature for the system \eqref{equequsysspatialsemischeme}, we have the fully discrete scheme
\begin{equation}\label{equfulldis}
	\left \{\begin{aligned}
		&\frac{G^n_{1,h}-G^{n-1}_{1,h}}{\tau}+a\sum_{i=0}^{n-1}d^{1-\alpha_1}_iG^{n-i}_{1,h}+
		\sum_{i=0}^{n-1}d^{1-\alpha_1}_iA_h G^{n-i}_{1,h}=a\sum_{i=0}^{n-1}d^{1-\alpha_2}_iG^{n-i}_{2,h}+f^n_{1,h},\\
		&\frac{G^n_{2,h}-G^{n-1}_{2,h}}{\tau}+a\sum_{i=0}^{n-1}d^{1-\alpha_2}_iG^{n-i}_{2,h}+
		\sum_{i=0}^{n-1}d^{1-\alpha_2}_iA_h G^{n-i}_{2,h}=a\sum_{i=0}^{n-1}d^{1-\alpha_1}_iG^{n-i}_{1,h}+f^n_{2,h},\\
		&G^0_{1,h}=G_{1,h}(0),\\
		&G^0_{2,h}=G_{2,h}(0),
	\end{aligned}\right .
\end{equation}
where
\begin{equation}\label{equweightdalpha}
	\sum_{i=0}^{\infty}d^\alpha_i\zeta^i=\left (\frac{\delta(\zeta)}{\tau}\right )^{\alpha},\quad 0<\alpha<1,
\end{equation}
and $ G^n_{1,h} $, $G^n_{2,h}$ are the numerical solutions of  $ G_{1} $, $G_{2}$ at time $t_n$ and $f^n_{1,h}=P_hf_1(t_n)$,   $f^n_{2,h}=P_hf_2(t_n)$.
\subsection{Error estimates for the homogeneous problem}
Here, we consider the error estimates when $f_1(t)=0$ and $f_2(t)=0$. To get the solutions of the system \eqref{equfulldis}, multiplying $\zeta^n$ and summing from $1$ to $\infty$ for the both sides of the first two equations in (\ref{equfulldis}) lead to
\begin{equation*}
\begin{aligned}
		&\sum_{n=1}^{\infty}\frac{\zeta^n G^n_{1,h}-\zeta^n G^{n-1}_{1,h}}{\tau}+a\sum_{n=1}^{\infty}\sum_{i=0}^{n-1}d^{1-\alpha_1}_i\zeta^nG^{n-i}_{1,h}
		+\sum_{n=1}^{\infty}\sum_{i=0}^{n-1}d^{1-\alpha_1}_i\zeta^nA_h G^{n-i}_{1,h}\\
		&\quad\quad\quad\quad\quad\quad\quad\quad\quad\quad\quad\quad\quad\quad\quad=a\sum_{n=1}^{\infty}\sum_{i=0}^{n-1}d^{1-\alpha_2}_i\zeta^nG^{n-i}_{2,h},\\
		&\sum_{n=1}^{\infty}\frac{\zeta^n G^n_{2,h}-\zeta^n G^{n-1}_{2,h}}{\tau}+a\sum_{n=1}^{\infty}\sum_{i=0}^{n-1}d^{1-\alpha_2}_i\zeta^nG^{n-i}_{2,h}+
		\sum_{n=1}^{\infty}\sum_{i=0}^{n-1}d^{1-\alpha_2}_i\zeta^nA_h G^{n-i}_{2,h}\\
		&\quad\quad\quad\quad\quad\quad\quad\quad\quad\quad\quad\quad\quad\quad\quad=a\sum_{n=1}^{\infty}\sum_{i=0}^{n-1}d^{1-\alpha_1}_i\zeta^nG^{n-i}_{1,h}.\\
	\end{aligned}
\end{equation*}
According to \eqref{equweightdalpha}, we have
\begin{equation*}
\begin{aligned}
		&\left (\frac{1-\zeta}{\tau}\right )\sum_{i=1}^{\infty}G^i_{1,h}\zeta^i+a\left (\frac{1-\zeta}{\tau}\right )^{1-\alpha_1}\sum_{i=1}^{\infty}G^i_{1,h}\zeta^i+\left (\frac{1-\zeta}{\tau}\right )^{1-\alpha_1}A_h\sum_{i=1}^{\infty}G^i_{1,h}\zeta^i\\=&a\left (\frac{1-\zeta}{\tau}\right )^{1-\alpha_2}\sum_{i=1}^{\infty}G^i_{2,h}\zeta^i+\frac{\zeta G_{1,h}(0)}{\tau},\\
		&\left (\frac{1-\zeta}{\tau}\right )\sum_{i=1}^{\infty}G^i_{2,h}\zeta^i+a\left (\frac{1-\zeta}{\tau}\right )^{1-\alpha_2}\sum_{i=1}^{\infty}G^i_{2,h}\zeta^i+\left (\frac{1-\zeta}{\tau}\right )^{1-\alpha_2}A_h\sum_{i=1}^{\infty}G^i_{2,h}\zeta^i\\=&a\left (\frac{1-\zeta}{\tau}\right )^{1-\alpha_1}\sum_{i=1}^{\infty}G^i_{1,h}\zeta^i+\frac{\zeta G_{2,h}(0)}{\tau},
	\end{aligned}
\end{equation*}
which results in, after simple calculations,

\begin{equation}\label{equnumsollapformtoest}
\begin{aligned}
\sum_{i=1}^{\infty}G^i_{1,h}\zeta^i=&\frac{\zeta}{\tau}\left (H_{\alpha_1,h}\left(\frac{1-\zeta}{\tau}\right)\left (\frac{1-\zeta}{\tau}\right )^{\alpha_1-1}G_{1,h}(0)\right.\\
&\left.\qquad+aH_h\left(\frac{1-\zeta}{\tau}\right)\left (\frac{1-\zeta}{\tau}\right )^{\alpha_1-1}G_{2,h}(0)\right ),\\
\sum_{i=1}^{\infty}G^i_{2,h}\zeta^i=&\frac{\zeta}{\tau}\left (aH_h\left(\frac{1-\zeta}{\tau}\right)\left (\frac{1-\zeta}{\tau}\right )^{\alpha_2-1}G_{1,h}(0)\right.\\
&\left.\quad+H_{\alpha_2,h}\left(\frac{1-\zeta}{\tau}\right)\left (\frac{1-\zeta}{\tau}\right )^{\alpha_2-1}G_{2,h}(0)\right ),
\end{aligned}
\end{equation}
	where $H_h$, $H_{\alpha_1,h}$, and $H_{\alpha_2,h}$ are defined by \eqref{Hh} and \eqref{Halpha}.
	
Now we give the error estimates of the solutions of the systems \eqref{equequsysspatialsemiAh} and \eqref{equfulldis} when $f_1=0$ and $f_2=0$.
\begin{theorem}\label{thmhomfullest}
	Let $G_{1,h}$, $G_{2,h}$ and $G^n_{1,h}$, $G^n_{2,h}$ be, respectively, the solutions of the systems \eqref{equequsysspatialsemiAh} and \eqref{equfulldis} with $f_1=0$, $f_2=0$. Then
	\begin{equation*}
	\begin{aligned}
		\|G_{1,h}(t_n)-G^n_{1,h}\|_{L^2(\Omega)}\leq C\tau(t_n^{-1}\|G_{1,h}(0)\|_{L^2(\Omega)}+t_n^{\alpha_2-1}\|G_{2,h}(0)\|_{L^2(\Omega)}),\\
		\|G_{2,h}(t_n)-G^n_{2,h}\|_{L^2(\Omega)}\leq C\tau(t_n^{\alpha_1-1}\|G_{1,h}(0)\|_{L^2(\Omega)}+t_n^{-1}\|G_{2,h}(0)\|_{L^2(\Omega)}).
	\end{aligned}
	\end{equation*}
\end{theorem}
	\begin{proof}
		We first consider the error estimates between $ G^n_{1,h} $ and $ G_{1,h}(t_n) $. By \eqref{equnumsollapformtoest}, for small $\xi_\tau=e^{-\tau(\kappa+1)}$, there is
		\begin{equation*}
			\begin{aligned}
			G^n_{1,h}=&\frac{1}{2\pi i\tau}\int_{|\zeta|=\xi_\tau}\zeta^{-n-1}\zeta\times\left (H_{\alpha_1,h}\left(\frac{1-\zeta}{\tau}\right)\left (\frac{1-\zeta}{\tau}\right )^{\alpha_1-1}G_{1,h}(0)\right.\\&\qquad\qquad\qquad\qquad\qquad\left.+aH_{h}\left(\frac{1-\zeta}{\tau}\right)\left (\frac{1-\zeta}{\tau}\right )^{\alpha_1-1}G_{2,h}(0)\right )d\zeta.
			\end{aligned}
		\end{equation*}
		Taking $\zeta=e^{-z\tau}$, we obtain
		\begin{equation*}
			\begin{aligned}
			G^n_{1,h}=\frac{1}{2\pi i}&\int_{\Gamma^\tau}e^{zt_n}e^{-z\tau}\times\left (H_{\alpha_1,h}\left(\frac{1-e^{-z\tau}}{\tau}\right)\left (\frac{1-e^{-z\tau}}{\tau}\right )^{\alpha_1-1}G_{1,h}(0)\right.\\&\qquad\qquad\qquad\left.+aH_{h}\left(\frac{1-e^{-z\tau}}{\tau}\right)\left (\frac{1-e^{-z\tau}}{\tau}\right )^{\alpha_1-1}G_{2,h}(0)\right )dz,
			\end{aligned}
		\end{equation*}
 where $\Gamma^\tau=\{z=\kappa+1+iy:y\in\mathbb{R}~{\rm and}~|y|\leq \pi/\tau\}$. Next we deform the contour $\Gamma^\tau$ to
$\Gamma^\tau_{\theta,\kappa}=\{z\in \mathbb{C}:\kappa\leq |z|\leq\frac{\pi}{\tau\sin(\theta)},|\arg z|=\theta\}\bigcup\{z\in \mathbb{C}:|z|=\kappa,|\arg z|\leq\theta\}$. Thus
\begin{equation}\label{equfulldissolG1}
\begin{aligned}
G^n_{1,h}=&\frac{1}{2\pi i}\int_{\Gamma^\tau_{\theta,\kappa}}e^{zt_n}e^{-z\tau}\times\left (H_{\alpha_1,h}\left(\frac{1-e^{-z\tau}}{\tau}\right)\left (\frac{1-e^{-z\tau}}{\tau}\right )^{\alpha_1-1}G_{1,h}(0)\right.\\&\left.\qquad\qquad\qquad\qquad+aH_{h}\left(\frac{1-e^{-z\tau}}{\tau}\right)\left (\frac{1-e^{-z\tau}}{\tau}\right )^{\alpha_1-1}G_{2,h}(0)\right )dz.
\end{aligned}
\end{equation}
In view of \eqref{equequsysspatialsemilapform}, $ f_1=0 $,  and $f_2=0$, there exists
\begin{equation*}
\tilde{G}_{1,h}=z^{\alpha_1-1}\left (H_{\alpha_1,h}(z)G_{1,h}(0)+aH_{h}(z)G_{2,h}(0)\right ).
\end{equation*}
So
\begin{equation}\label{eqG1}
	G_{1,h}(t)=\frac{1}{2\pi i}\int_{\Gamma_{\theta,\kappa}}e^{zt}\left (H_{\alpha_1,h}(z)z^{\alpha_1-1}G_{1,h}(0)+aH_h(z)z^{\alpha_1-1}G_{2,h}(0)\right )dz.
\end{equation}
Combining \eqref{equfulldissolG1} and \eqref{eqG1} leads to
\begin{equation*}
	\begin{aligned}
		&G_{1,h}(t_n)-G^n_{1,h}\\=&\frac{1}{2\pi i}\int_{\Gamma_{\theta,\kappa}\backslash\Gamma^{\tau}_{\theta,\kappa}}e^{zt_n} H_{\alpha_1,h}(z)z^{\alpha_1-1}G_{1,h}(0)dz+\frac{1}{2\pi i}\int_{\Gamma_{\theta,\kappa}\backslash\Gamma^{\tau}_{\theta,\kappa}}e^{zt_n}aH_h(z)z^{\alpha_1-1}\\
		& \cdot G_{2,h}(0)dz+\frac{1}{2\pi i}\int_{\Gamma^{\tau}_{\theta,\kappa}}e^{zt_n}\left (H_{\alpha_1,h}(z)z^{\alpha_1-1}-e^{-z\tau}H_{\alpha_1,h}\left(\frac{1-e^{-z\tau}}{\tau}\right) \right. \\
		&
		\left.
		\cdot\left(\frac{1-e^{-z\tau}}{\tau}\right)^{\alpha_1-1}\right )G_{1,h}(0)dz+\frac{a}{2\pi i}\int_{\Gamma^{\tau}_{\theta,\kappa}}e^{zt_n}\left(H_h(z)z^{\alpha_1-1}-e^{-z\tau}\left(\frac{1-e^{-z\tau}}{\tau}\right) \right.
		\\
		&
		\left. \cdot\left(\frac{1-e^{-z\tau}}{\tau}\right)^{\alpha_1-1}H_{h}\right)G_{2,h}(0)dz\\
		=&\uppercase\expandafter{\romannumeral1}+\uppercase\expandafter{\romannumeral2}+\uppercase\expandafter{\romannumeral3}+\uppercase\expandafter{\romannumeral4}.
	\end{aligned}
\end{equation*}
According to Lemma \ref{lemtheestimateofHhblabla}, we have
\begin{equation*}
\begin{aligned}
	\|\uppercase\expandafter{\romannumeral1}\|_{L^2(\Omega)}&\leq C\int_{\Gamma_{\theta,\kappa}\backslash\Gamma^{\tau}_{\theta,\kappa}}e^{-C|z|t_n}|z|^{\alpha_1-1} \|H_{\alpha_1,h}(z)\||dz|\|G_{1,h}(0)\|_{L^2(\Omega)}\\&\leq C t_n^{-1}\tau\|G_{1,h}(0)\|_{L^2(\Omega)}.
\end{aligned}
\end{equation*}
For $ \uppercase\expandafter{\romannumeral2} $, similarly, there exists
\begin{equation*}
\begin{aligned}
	\|\uppercase\expandafter{\romannumeral2}\|_{L^2(\Omega)}&\leq C\int_{\Gamma_{\theta,\kappa}\backslash\Gamma^\tau_{\theta,\kappa}}e^{-C|z|t_n}|z|^{\alpha_1-1}a\|H_h(z)\||dz|\|G_{2,h}(0)\|_{L^2(\Omega)}\\&\leq C t_n^{\alpha_2-1}\tau\|G_{2,h}(0)\|_{L^2(\Omega)}.
\end{aligned}
\end{equation*}
Next for $ \uppercase\expandafter{\romannumeral3} $ and $ \uppercase\expandafter{\romannumeral4} $, we obtain
\begin{equation*}
	\begin{aligned}
		\uppercase\expandafter{\romannumeral3}=&\frac{1}{2\pi i}\int_{\Gamma^{\tau}_{\theta,\kappa}}e^{zt_n}e^{-z\tau}\left (e^{z\tau}H_{\alpha_1,h}(z)z^{\alpha_1-1}-H_{\alpha_1,h}\left(\frac{1-e^{-z\tau}}{\tau}\right)\right.
		\\
		&	\left. \cdot\left(\frac{1-e^{-z\tau}}{\tau}\right)^{\alpha_1-1}\right )G_{1,h}(0)dz\\
		=
		&	
		\frac{1}{2\pi i}\int_{\Gamma^{\tau}_{\theta,\kappa}}e^{zt_n}e^{-z\tau}\left ( H_{\alpha_1,h}(z)z^{\alpha_1-1}-H_{\alpha_1,h}\left(\frac{1-e^{-z\tau}}{\tau}\right)\left(\frac{1-e^{-z\tau}}{\tau}\right)^{\alpha_1-1}\right )
		\\
		&\cdot G_{1,h}(0)dz+\frac{1}{2\pi i}\int_{\Gamma^{\tau}_{\theta,\kappa}}e^{zt_n}e^{-z\tau} (e^{z\tau}-1)H_{\alpha_1,h}(z)z^{\alpha_1-1}G_{1,h}(0)dz
		\\
		=&\uppercase\expandafter{\romannumeral3}_1+\uppercase\expandafter{\romannumeral3}_2
	\end{aligned}
\end{equation*}
and
\begin{equation*}
	\begin{aligned}
	\uppercase\expandafter{\romannumeral4}=&\frac{a}{2\pi i}\int_{\Gamma^{\tau}_{\theta,\kappa}}e^{zt_n}e^{-z\tau}\left(e^{z\tau}H_h(z)z^{\alpha_1-1}-H_{h}\left(\frac{1-e^{-z\tau}}{\tau}\right)\left(\frac{1-e^{-z\tau}}{\tau}\right)^{\alpha_1-1}\right)\\
	&
\cdot G_{2,h}(0)dz
\\
	=&
\frac{a}{2\pi i}\int_{\Gamma^{\tau}_{\theta,\kappa}}e^{zt_n}e^{-z\tau}\left(H_h(z)z^{\alpha_1-1}-H_{h}\left(\frac{1-e^{-z\tau}}{\tau}\right)\left(\frac{1-e^{-z\tau}}{\tau}\right)^{\alpha_1-1}\right)\\
	&\cdot G_{2,h}(0)dz+\frac{a}{2\pi i}\int_{\Gamma^{\tau}_{\theta,\kappa}}e^{zt_n}e^{-z\tau}(e^{z\tau}-1)H_h(z)z^{\alpha_1-1}G_{2,h}(0)dz=\uppercase\expandafter{\romannumeral4}_1+\uppercase\expandafter{\romannumeral4}_2.
	\end{aligned}
\end{equation*}
 Using $ \|\frac{d}{dz}(H_{\alpha_1,h}(z)z^{\alpha_1-1})\|\leq
  C|z|^{-2} $, $\|\frac{d}{dz}(H_{h}(z)z^{\alpha_1-1})\|\leq
  C|z|^{-\alpha_2-2}$, the mean value theorem, and the fact $\left(\frac{1-e^{-z\tau}}{\tau}\right)=z+O(\tau z^2)$, the estimates
\begin{equation*}
\left \|H_{\alpha_1,h}(z)z^{\alpha_1-1}-H_{\alpha_1,h}\left(\frac{1-e^{-z\tau}}{\tau}\right )\left(\frac{1-e^{-z\tau}}{\tau}\right)^{\alpha_1-1}\right \|\leq C\tau
\end{equation*}
and
\begin{equation*}
	\left \|H_h(z)z^{\alpha_1-1}-H_{h}\left(\frac{1-e^{-z\tau}}{\tau}\right )\left(\frac{1-e^{-z\tau}}{\tau}\right)^{\alpha_1-1}\right \|\leq C\tau|z|^{-\alpha_2}
\end{equation*}
can be obtained. Thus
\begin{equation*}
	\|\uppercase\expandafter{\romannumeral3}_1\|_{L^2(\Omega)}\leq C\tau\int_{\Gamma^{\tau}_{\theta,\kappa}}e^{-C|z|t_{n-1}}|dz|\|G_{1,h}(0)\|_{L^2(\Omega)}\leq C t_n^{-1}\tau\|G_{1,h}(0)\|_{L^2(\Omega)}
\end{equation*}
and
\begin{equation*}
\|\uppercase\expandafter{\romannumeral4}_1\|_{L^2(\Omega)}\leq C\tau\int_{\Gamma^{\tau}_{\theta,\kappa}}e^{-C|z|t_{n-1}}|z|^{-\alpha_2}|dz|\|G_{2,h}(0)\|_{L^2(\Omega)}\leq C t_n^{\alpha_2-1}\tau\|G_{2,h}(0)\|_{L^2(\Omega)}.
\end{equation*}
Using $|e^{z\tau}-1|\leq C\tau|z|$ and Lemma \ref{lemtheestimateofHhblabla}, we have
\begin{equation*}
	\begin{aligned}
	&\|\uppercase\expandafter{\romannumeral3}_2\|_{L^2(\Omega)}\leq C\tau\int_{\Gamma^{\tau}_{\theta,\kappa}}e^{-C|z|t_{n-1}} |dz|\|G_{1,h}(0)\|_{L^2(\Omega)}\leq Ct_n^{-1}\tau\|G_{1,h}(0)\|_{L^2(\Omega)},\\
	&\|\uppercase\expandafter{\romannumeral4}_2\|_{L^2(\Omega)}\leq C\tau\int_{\Gamma^{\tau}_{\theta,\kappa}}e^{-C|z|t_{n-1}} |z|^{-\alpha_2}|dz|\|G_{2,h}(0)\|_{L^2(\Omega)}\leq Ct_n^{\alpha_2-1}\tau\|G_{2,h}(0)\|_{L^2(\Omega)}.\\
	\end{aligned}
\end{equation*}
In summary,
\begin{equation*}
	\|G_{1,h}(t_n)-G^n_{1,h}\|_{L^2(\Omega)}\leq C\tau(t_n^{-1}\|G_{1,h}(0)\|_{L^2(\Omega)}+t_n^{\alpha_2-1}\|G_{2,h}(0)\|_{L^2(\Omega)}).
\end{equation*}
Analogously, we have
\begin{equation*}
\|G_{2,h}(t_n)-G^n_{2,h}\|_{L^2(\Omega)}\leq C\tau(t_n^{\alpha_1-1}\|G_{1,h}(0)\|_{L^2(\Omega)}+t_n^{-1}\|G_{2,h}(0)\|_{L^2(\Omega)}).
\end{equation*}
The proof has been completed.
\end{proof}
Combining Theorem \ref{thmsmoothdatasemi}, Theorem \ref{thmnonsmoothdatasemi}, and Theorem \ref{thmhomfullest}, we have the error estimates for homogeneous problem.
\begin{theorem}
	Let $G_{1}$, $G_{2}$ and $G^n_{1,h}$, $G^n_{2,h}$ be, respectively, the solutions of the systems \eqref{equrqtosol} and \eqref{equfulldis}  with $f_1=0$, $f_2=0$. Then we have estimates
	\begin{itemize}
		\item if $G_{1,0}$, $G_{2,0}\in \dot{H}^2(\Omega)$,
		\begin{equation*}
		\begin{aligned}
		\|G_{1}(t_n)-G^n_{1,h}\|_{L^2(\Omega)}\leq&Ch^2\left (\|G_{1,0}\|_{\dot{H}^2(\Omega)}+\|G_{2,0}\|_{\dot{H}^2(\Omega)}\right )\\ &+C\tau(t_n^{-1}\|G_{1,h}(0)\|_{L^2(\Omega)}+t_n^{\alpha_2-1}\|G_{2,h}(0)\|_{L^2(\Omega)}),\\
		\|G_{2}(t_n)-G^n_{2,h}\|_{L^2(\Omega)}\leq &Ch^2\left (\|G_{1,0}\|_{\dot{H}^2(\Omega)}+\|G_{2,0}\|_{\dot{H}^2(\Omega)}\right )\\
		&+C\tau(t_n^{\alpha_1-1}\|G_{1,h}(0)\|_{L^2(\Omega)}+t_n^{-1}\|G_{2,h}(0)\|_{L^2(\Omega)});
		\end{aligned}
		\end{equation*}
		\item if $G_{1,0}$, $G_{2,0}\in L^2(\Omega)$,
		\begin{equation*}
		\begin{aligned}
		\|G_{1}(t_n)-G^n_{1,h}\|_{L^2(\Omega)}\leq&Ch^2(t^{-\alpha_1}\|G_{1,0}\|_{L^2(\Omega)}+\|G_{2,0}\|_{L^2(\Omega)})\\ &+C\tau(t_n^{-1}\|G_{1,h}(0)\|_{L^2(\Omega)}+t_n^{\alpha_2-1}\|G_{2,h}(0)\|_{L^2(\Omega)}),\\
		\|G_{2}(t_n)-G^n_{2,h}\|_{L^2(\Omega)}\leq&Ch^2(\|G_{1,0}\|_{L^2(\Omega)}+t^{-\alpha_2}\|G_{2,0}\|_{L^2(\Omega)})\\
		&+ C\tau(t_n^{\alpha_1-1}\|G_{1,h}(0)\|_{L^2(\Omega)}+t_n^{-1}\|G_{2,h}(0)\|_{L^2(\Omega)}).
		\end{aligned}
		\end{equation*}
	\end{itemize}
	
\end{theorem}
\subsection{Error estimates for the inhomogeneous problem}
Now we consider the error estimates for the inhomogeneous problem with vanishing initial value.
Multiplying $\zeta^n$ and summing from $0$ to $\infty$ on both sides of the first two equations in \eqref{equfulldis} result in
\begin{equation*}
\begin{aligned}
&\sum_{n=1}^{\infty}\frac{\zeta^n G^n_{1,h}-\zeta^n G^{n-1}_{1,h}}{\tau}+a\sum_{n=1}^{\infty}\sum_{i=0}^{n-1}d^{1-\alpha_1}_i\zeta^nG^{n-i}_{1,h}\\
+&
\sum_{n=1}^{\infty}\sum_{i=0}^{n-1}d^{1-\alpha_1}_i\zeta^nA_h G^{n-i}_{1,h}
=a\sum_{n=1}^{\infty}\sum_{i=0}^{n-1}d^{1-\alpha_2}_i\zeta^nG^{n-i}_{2,h}+\sum_{n=1}^{\infty}\zeta^nf^n_{1,h},\\
&\sum_{n=1}^{\infty}\frac{\zeta^n G^n_{2,h}-\zeta^n G^{n-1}_{2,h}}{\tau}+a\sum_{n=1}^{\infty}\sum_{i=0}^{n-1}d^{1-\alpha_2}_i\zeta^nG^{n-i}_{2,h}\\+&
\sum_{n=1}^{\infty}\sum_{i=0}^{n-1}d^{1-\alpha_2}_i\zeta^nA_h G^{n-i}_{2,h}
=a\sum_{n=1}^{\infty}\sum_{i=0}^{n-1}d^{1-\alpha_1}_i\zeta^nG^{n-i}_{1,h}+\sum_{n=1}^{\infty}\zeta^nf^n_{2,h}.\\
\end{aligned}
\end{equation*}
Using the property of $d^{\alpha}_i$ in \eqref{equweightdalpha} and the vanishing initial value, we have
\begin{equation*}
\begin{aligned}
&\left (\frac{1-\zeta}{\tau}\right )\sum_{i=0}^{\infty}G^i_{1,h}\zeta^i+a\left (\frac{1-\zeta}{\tau}\right )^{1-\alpha_1}\sum_{i=0}^{\infty}G^i_{1,h}\zeta^i+\left (\frac{1-\zeta}{\tau}\right )^{1-\alpha_1}A_h\sum_{i=0}^{\infty}G^i_{1,h}\zeta^i\\=&a\left (\frac{1-\zeta}{\tau}\right )^{1-\alpha_2}\sum_{i=0}^{\infty}G^i_{2,h}\zeta^i+\sum_{i=1}^{\infty}\zeta^if^i_{1,h},\\
&\left (\frac{1-\zeta}{\tau}\right )\sum_{i=0}^{\infty}G^i_{2,h}\zeta^i+a\left (\frac{1-\zeta}{\tau}\right )^{1-\alpha_2}\sum_{i=0}^{\infty}G^i_{2,h}\zeta^i+\left (\frac{1-\zeta}{\tau}\right )^{1-\alpha_2}A_h\sum_{i=0}^{\infty}G^i_{2,h}\zeta^i\\=&a\left (\frac{1-\zeta}{\tau}\right )^{1-\alpha_1}\sum_{i=0}^{\infty}G^i_{1,h}\zeta^i+\sum_{i=1}^{\infty}\zeta^if^i_{2,h}.
\end{aligned}
\end{equation*}
Thus, simple calculation leads to
\begin{equation}\label{equnumsollapformtoestG00}
\begin{aligned}
\sum_{i=0}^{\infty}G^i_{1,h}\zeta^i=&H_{\alpha_1,h}\left(\frac{1-\zeta}{\tau}\right)\left (\frac{1-\zeta}{\tau}\right )^{\alpha_1-1}\sum_{i=1}^{\infty}\zeta^if^i_{1,h}\\
&+aH_h\left(\frac{1-\zeta}{\tau}\right)\left (\frac{1-\zeta}{\tau}\right )^{\alpha_1-1}\sum_{i=1}^{\infty}\zeta^if^i_{2,h},\\
\sum_{i=0}^{\infty}G^i_{2,h}\zeta^i=&aH_h\left(\frac{1-\zeta}{\tau}\right)\left (\frac{1-\zeta}{\tau}\right )^{\alpha_2-1}\sum_{i=1}^{\infty}\zeta^if^i_{1,h}\\
&+H_{\alpha_2,h}\left(\frac{1-\zeta}{\tau}\right)\left (\frac{1-\zeta}{\tau}\right )^{\alpha_2-1}\sum_{i=1}^{\infty}\zeta^if^i_{2,h}.
\end{aligned}
\end{equation}

Then we have the error estimates of the solutions of the systems \eqref{equequsysspatialsemiAh} and \eqref{equfulldis} when $G_{1,0}=0$, $G_{2,0}=0$.
\begin{theorem}\label{thminhomfullest}
	Let $G_{1,h}$, $G_{2,h}$ and $G^n_{1,h}$, $G^n_{2,h}$ be the solutions of the systems \eqref{equequsysspatialsemiAh} and \eqref{equfulldis} with $G_{1,0}=0$, $G_{2,0}=0$ and $\int_{0}^{t_n}\|f_1'(s)\|_{L^2(\Omega)}ds<\infty$, $\int_{0}^{t_n}\|f_2'(s)\|_{L^2(\Omega)}ds<\infty$.  Then
	\begin{equation*}
	\begin{aligned}
	\|G_{1,h}(t_n)-G^n_{1,h}\|_{L^2(\Omega)}\leq& C\tau \left(\int_{0}^{t_n}\|f_1'(s)\|_{L^2(\Omega)}ds+\|f_1(0)\|_{L^2(\Omega)}\right.\\&\left.+\int_{0}^{t_n}\|f_2'(s)\|_{L^2(\Omega)}ds+\|f_2(0)\|_{L^2(\Omega)}\right),\\
	\|G_{2,h}(t_n)-G^n_{2,h}\|_{L^2(\Omega)}\leq& C\tau \left(\int_{0}^{t_n}\|f_1'(s)\|_{L^2(\Omega)}dt+\|f_1(0)\|_{L^2(\Omega)}\right.\\
	&\left.+\int_{0}^{t_n}\|f_2'(s)\|_{L^2(\Omega)}ds+\|f_2(0)\|_{L^2(\Omega)}\right).
	\end{aligned}
	\end{equation*}
\end{theorem}
\begin{proof}
	We just give the error estimate between $ G^n_{1,h} $ and $ G_{1,h}(t_n) $. Denote
	\begin{equation*}
	aH_h\left(\frac{1-\zeta}{\tau}\right)\left (\frac{1-\zeta}{\tau}\right )^{\alpha_1-1}=\sum_{j=0}^{\infty}E_{1,j}\zeta^j 
	\end{equation*}
and
\begin{equation*}
H_{\alpha_1,h}\left(\frac{1-\zeta}{\tau}\right)\left (\frac{1-\zeta}{\tau}\right )^{\alpha_1-1}=\sum_{j=0}^{\infty}E_{\alpha_1,j}\zeta^j.
\end{equation*}
	Then \eqref{equnumsollapformtoestG00} can be rewritten as
	\begin{equation*}
	\sum_{i=0}^{\infty}G^i_{1,h}\zeta^i=\left (\sum_{j=0}^{\infty}E_{\alpha_1,j}\zeta^j\right)\left (\sum_{i=1}^{\infty}\zeta^if^i_{1,h}\right )+\left (\sum_{j=0}^{\infty}E_{1,j}\zeta^j\right)\left (\sum_{i=1}^{\infty}\zeta^if^i_{2,h}\right ).
	\end{equation*}
	Thus
	\begin{equation}\label{equnumerG001repcon1}
		G^n_{1,h}=\lim_{t\rightarrow t_n^{-}}\left(\left (\hat{E}_{\alpha_1}\ast f_{1,h}\right )(t)+\left (\hat{E}_{1}\ast f_{2,h}\right )(t)\right),
	\end{equation}
	where $ \hat{E}_{\alpha_1}=\sum_{j=0}^{\infty}E_{\alpha_1,j}\delta_{t_j} $ and $ \hat{E}_{1}=\sum_{j=0}^{\infty}E_{1,j}\delta_{t_j} $ with $ \delta_{t} $ the delta function concentrated at $t$. Introduce
	\begin{equation*}
		Q_{\alpha_1}(t',v)=\hat{E}_{\alpha_1}\ast v,\quad Q_{1}(t',v)=\hat{E}_{1}\ast v.
	\end{equation*}
	 The fact $f_1(t)=f_1(0)+1\ast f_1'(t)$, $f_2(t)=f_2(0)+1\ast f_2'(t)$, and \eqref{equnumerG001repcon1} leads to
	 \begin{equation}\label{equG00fullsollap}
	 \begin{aligned}
		 G^n_{1,h}=&\lim_{t\rightarrow t_n^{-}}\left((Q_{\alpha_1}(t',1)\ast f_1')(t)+(Q_{1}(t',1)\ast f_2')(t)\right)\\
		 &+\lim_{t\rightarrow t_n^{-}}\left(Q_{\alpha_1}(t,1) f_1(0)+Q_{1}(t,1) f_2(0)\right).
	 \end{aligned}
	 \end{equation}
	 Similarly,
	 \begin{equation}\label{equG00semisollap}
	 \begin{aligned}
	 G_{1,h}(t)=&\frac{1}{2\pi i}\left (\int_{\Gamma_{\theta,\kappa}}e^{zt'}H_{\alpha_1,h}(z)z^{\alpha_1-1}z^{-1}dz\ast f_1'(t')\right )(t)\\&+\frac{a}{2\pi i}\left (\int_{\Gamma_{\theta,\kappa}}e^{zt'}H_{h}(z)z^{\alpha_1-1}z^{-1}dz\ast f_2'(t')\right )(t)\\
	 &+\frac{1}{2\pi i}\int_{\Gamma_{\theta,\kappa}}e^{zt}H_{\alpha_1,h}(z)z^{\alpha_1-1}z^{-1}dz f_1(0)\\&+\frac{a}{2\pi i}\int_{\Gamma_{\theta,\kappa}}e^{zt}H_{h}(z)z^{\alpha_1-1}z^{-1}dz f_2(0)
	 \end{aligned}
	 \end{equation}
	 can be obtained from \eqref{equequsysspatialsemilapform}. Combining \eqref{equG00fullsollap} and \eqref{equG00semisollap} results in
	 \begin{equation*}
	 \begin{aligned}
	 G_{1,h}(t_n)-G^n_{1,h}=&\lim_{t\rightarrow t_n^{-}}\left(\left (\uppercase\expandafter{\romannumeral1}(t')\ast f_1'(t')\right )(t)+\left (\uppercase\expandafter{\romannumeral2}(t')\ast f_2'(t')\right )(t)\right)\\
	 &+\lim_{t\rightarrow t_n^{-}}\left(\uppercase\expandafter{\romannumeral1}(t) f_1(0)+\uppercase\expandafter{\romannumeral2}(t) f_2(0)\right),
	 \end{aligned}
	 \end{equation*}
	 where
	 \begin{equation*}
	 	\begin{aligned}
	 		&\uppercase\expandafter{\romannumeral1}(t)=\frac{1}{2\pi i} \int_{\Gamma_{\theta,\kappa}}e^{zt}H_{\alpha_1,h}(z)z^{\alpha_1-1}z^{-1}dz-Q_{\alpha_1}(t,1),\\
	 		&\uppercase\expandafter{\romannumeral2}(t)=\frac{1}{2\pi i} \int_{\Gamma_{\theta,\kappa}}e^{zt}H_{h}(z)z^{\alpha_1-1}z^{-1}dz-Q_{1}(t,1).\\
	 	\end{aligned}
	 \end{equation*}
	 For $t\in[t_{n-1},t_n)$, when $ n=1$, we have
	 \begin{equation*}
	 	\|\uppercase\expandafter{\romannumeral1}(t)\|\leq C\tau.
	 \end{equation*}
	 As for $n>1$, we take $\kappa\geq1/t^{n-1}$ and ensure that $\kappa$ is large enough to satisfy the conditions in Lemma \ref{lemtheestimateofHhblabla}. Then we have
	 \begin{equation*}
	 	\begin{aligned}
	 		\|\uppercase\expandafter{\romannumeral1}\|=&\Big \|\frac{1}{2\pi i} \int_{\Gamma_{\theta,\kappa}}e^{zt}H_{\alpha_1,h}(z)z^{\alpha_1-1}z^{-1}dz-\frac{1}{2\pi i} \int_{\Gamma_{\theta,\kappa}}e^{zt_{n-1}}H_{\alpha_1,h}(z)z^{\alpha_1-1}z^{-1}dz\Big \|\\
	 		&+\Big \|\frac{1}{2\pi i} \int_{\Gamma_{\theta,\kappa}}e^{zt_{n-1}}H_{\alpha_1,h}(z)z^{\alpha_1-1}z^{-1}dz-Q_{\alpha_1}(t,1)\Big \|\\
	 		\leq&C\tau+\Big \|\frac{1}{2\pi i} \int_{\Gamma_{\theta,\kappa}}e^{zt_{n-1}}H_{\alpha_1,h}(z)z^{\alpha_1-1}z^{-1}dz-Q_{\alpha_1}(t,1)\Big \|.	 		
	 	\end{aligned}
	 \end{equation*}
	 Taking $\xi_\tau=e^{-\tau(\kappa+1)}$, we obtain
	 \begin{equation*}
	 	E_{\alpha_1,n}=\frac{1}{2\pi i}\int_{|\zeta|=\xi_\tau}
	 	\zeta^{-n-1}H_{\alpha_1,h}\left(\frac{1-\zeta}{\tau}\right)\left (\frac{1-\zeta}{\tau}\right )^{\alpha_1-1}d\zeta.
	 \end{equation*}
	 Then
	 \begin{equation*}
	 	Q_{\alpha_1}(t,1)=\sum_{j=0}^{n-1}E_{\alpha_1,j}=\frac{1}{2\pi\tau i}\int_{|\zeta|=\xi_\tau}\zeta^{-n}H_{\alpha_1,h}\left(\frac{1-\zeta}{\tau}\right)\left (\frac{1-\zeta}{\tau}\right )^{\alpha_1-2}d\zeta
	 \end{equation*}
	 can be obtained from the fact $ \sum_{j=0}^{n-1}\ \zeta^{-j-1}=(\zeta^{-n}-1)/(1-\zeta)$ and for small $\zeta$, the term $ \left ((1-\zeta)/(\tau)\right )^{\alpha_1-1}H_{\alpha_1,h}\left((1-\zeta)/(\tau)\right)/(1-\zeta)$ is analytic. Taking $\zeta=e^{-z\tau}$, we get
	 \begin{equation*}
	 	Q_{\alpha_1}(t,1)=\frac{1}{2\pi i}\int_{\Gamma^\tau}e^{zt_{n-1}}H_{\alpha_1,h}\left(\frac{1-e^{-z\tau}}{\tau}\right)\left (\frac{1-e^{-z\tau}}{\tau}\right )^{\alpha_1-2}dz,
	 \end{equation*}
	  where $\Gamma^\tau=\{z=\kappa+1+iy:y\in\mathbb{R}~{\rm and}~|y|\leq \pi/\tau\}$. Next we deform the contour $\Gamma^\tau$ to
	 $\Gamma^\tau_{\theta,\kappa}=\{z\in \mathbb{C}:\kappa\leq |z|\leq\frac{\pi}{\tau\sin(\theta)},|\arg z|=\theta\}\bigcup\{z\in \mathbb{C}:|z|=\kappa,|\arg z|\leq\theta\}$. Thus
	 	 \begin{equation*}
	 Q_{\alpha_1}(t,1)=\frac{1}{2\pi i}\int_{\Gamma^{\tau}_{\theta,\kappa}}e^{zt_{n-1}}H_{\alpha_1,h}\left(\frac{1-e^{-z\tau}}{\tau}\right)\left (\frac{1-e^{-z\tau}}{\tau}\right )^{\alpha_1-2}dz,
	 \end{equation*}
	 which leads to
	 \begin{equation*}
	 	\begin{aligned}
	 		&\frac{1}{2\pi i} \int_{\Gamma_{\theta,\kappa}}e^{zt_{n-1}}H_{\alpha_1,h}(z)z^{\alpha_1-1}z^{-1}dz-Q_{\alpha_1}(t,1)\\
	 		=&\frac{1}{2\pi i} \int_{\Gamma_{\theta,\kappa}\backslash\Gamma^\tau_{\theta,\kappa}}e^{zt_{n-1}}H_{\alpha_1,h}(z)z^{\alpha_1-2}dz\\
	 		&+\frac{1}{2\pi i} \int_{\Gamma^\tau_{\theta,\kappa}}e^{zt_{n-1}}\left (H_{\alpha_1,h}(z)z^{\alpha_1-2}-H_{\alpha_1,h}\left(\frac{1-e^{-z\tau}}{\tau}\right)\left (\frac{1-e^{-z\tau}}{\tau}\right )^{\alpha_1-2}\right )dz\\
	 		=&\uppercase\expandafter{\romannumeral1}_1+\uppercase\expandafter{\romannumeral1}_2.
	 	\end{aligned}
	 \end{equation*}
	 For $\uppercase\expandafter{\romannumeral1}_1$, according to Lemma \ref{lemtheestimateofHhblabla}, we have the following estimate
	 \begin{equation*}
	 	\|\uppercase\expandafter{\romannumeral1}_1\|\leq C \int_{\Gamma_{\theta,\kappa}\backslash\Gamma^\tau_{\theta,\kappa}}e^{-C|z|t_{n-1}}|z|^{-2}|dz|\leq C\tau.
	 \end{equation*}
	 For $\uppercase\expandafter{\romannumeral1}_2$, using $ \left \|\frac{d}{dz}\left (H_{\alpha_1,h}(z)z^{\alpha_1-2}\right )\right \|\leq C|z|^{-3} $, the mean value theorem, and the fact $\left(\frac{1-e^{-z\tau}}{\tau}\right)=z+O(\tau z^2)$, we obtain
	 \begin{equation*}
	 	\Big \|H_{\alpha_1,h}(z)z^{\alpha_1-2}-H_{\alpha_1,h}\Big(\frac{1-e^{-z\tau}}{\tau}\Big)\Big (\frac{1-e^{-z\tau}}{\tau}\Big )^{\alpha_1-2}\Big\|\leq C|z|^{-3}|\tau z^2|\leq C\tau|z|^{-1}.
	 \end{equation*}
	 So
	 \begin{equation*}
	 	\|\uppercase\expandafter{\romannumeral1}_2\|\leq C\tau\int_{\Gamma^\tau_{\theta,\kappa}}e^{-C|z|t_{n-1}}|z|^{-1}|dz|\leq C\tau.
	 \end{equation*}
	 Consequently, we have
	 \begin{equation*}
	 	\|\uppercase\expandafter{\romannumeral1}\|\leq C\tau.
	 \end{equation*}
	 Similarly,
	 \begin{equation*}
	 	\|\uppercase\expandafter{\romannumeral2}\|\leq C\tau.
	 \end{equation*}
	Therefore, we get
	 \begin{equation*}
	 \begin{aligned}
	 	\|G_{1,h}(t_n)-G^n_{1,h}\|_{L^2(\Omega)}\leq& C\tau \left((1*\|f_1'\|_{L^2(\Omega)})(t_n)+(1*\|f_2'\|_{L^2(\Omega)})(t_n)\right.\\
	 	&\left.+\|f_1(0)\|_{L^2(\Omega)}+\|f_2(0)\|_{L^2(\Omega)}\right).
	 \end{aligned}
	 \end{equation*}
	 Also, we can obtain
	  \begin{equation*}
	  \begin{aligned}
	 \|G_{2,h}(t_n)-G^n_{2,h}\|_{L^2(\Omega)}\leq& C\tau \left((1*\|f_1'\|_{L^2(\Omega)})(t_n)+(1*\|f_2'\|_{L^2(\Omega)})(t_n)\right.\\&+\left.\|f_1(0)\|_{L^2(\Omega)}+\|f_2(0)\|_{L^2(\Omega)}\right).
	 \end{aligned}
	 \end{equation*}
\end{proof}

Lastly, Theorems \ref{thmvanisheddatasemi} and \ref{thminhomfullest} lead to the error estimates.
\begin{theorem}
	Let $G_{1}$, $G_{2}$ and $G^n_{1,h}$, $G^n_{2,h}$ be the solutions of the systems \eqref{equrqtosol} and \eqref{equfulldis} with $G_{1,0}=0$, $G_{2,0}=0$, $\int_{0}^{t_n}\|f_1(s)\|_{L^2(\Omega)}ds<\infty$, $\int_{0}^{t_n}\|f_2(s)\|_{L^2(\Omega)}ds<\infty$,  $\int_{0}^{t_n}\|f_1'(s)\|_{L^2(\Omega)}ds<\infty$, and $\int_{0}^{t_n}\|f_2'(s)\|_{L^2(\Omega)}ds<\infty$.  Then
	\begin{equation*}
	\begin{aligned}
	\|G_{1}(t_n)-G^n_{1,h}\|_{L^{2}(\Omega)}\leq&Ch^2\left (\int_0^t (t-s)^{-\alpha_1}\|f_1(s)\|_{L^2(\Omega)}ds+\int_0^t\|f_2(s)\|_{L^2(\Omega)}ds\right )\\
	&+C\tau \left(\int_{0}^{t_n}\|f_1'(s)\|_{L^2(\Omega)}ds+\|f_1(0)\|_{L^2(\Omega)}\right.\\&\left.+\int_{0}^{t_n}\|f_2'(s)\|_{L^2(\Omega)}ds+\|f_2(0)\|_{L^2(\Omega)}\right),\\
	\|G_{2}(t_n)-G^n_{2,h}\|_{L^2(\Omega)}\leq&Ch^2\left (\int_0^t\|f_1(s)\|_{L^2(\Omega)}ds+\int_0^t(t-s)^{-\alpha_2}\|f_2(s)\|_{L^2(\Omega)}ds\right )\\
	 &+C\tau \left(\int_{0}^{t_n}\|f_1'(s)\|_{L^2(\Omega)}dt+\|f_1(0)\|_{L^2(\Omega)}\right.\\
	&\left.+\int_{0}^{t_n}\|f_2'(s)\|_{L^2(\Omega)}ds+\|f_2(0)\|_{L^2(\Omega)}\right).
	\end{aligned}
	\end{equation*}
\end{theorem}

\section{Numerical experiments}
In this section, we perform the one- and two-dimensional numerical experiments to verify the effectiveness of the numerical schemes. Here, we let
\begin{equation*}
\begin{aligned}
E_{1,h}=E_{1,\tau}=G_{1}(t_n)-G^{n}_{1,h},\quad
E_{2,h}=E_{2,\tau}=G_{2}(t_n)-G^{n}_{2,h},
\end{aligned}
\end{equation*}
if the exact solutions $G_1$ and $G_2$ are known. If the exact solutions $G_1$ and $G_2$ are unknown, to get the spatial errors, denote
\begin{equation*}
\begin{aligned}
E_{1,h}=G^{n}_{1,h}-G^{n}_{1,h/2},\quad
E_{2,h}=G^{n}_{2,h}-G^{n}_{2,h/2},
\end{aligned}
\end{equation*}
where the $G^n_{1,h}$ and $G^n_{2,h}$ mean the numerical solutions of $G_{1}$ and $G_{2}$ at $t_n$ with mesh size $h$; similarly, to get the temporal errors, we let
\begin{equation*}
\begin{aligned}
E_{1,\tau}=G_{1,\tau}-G_{1,\tau/2},\quad
E_{2,\tau}=G_{2,\tau}-G_{2,\tau/2},
\end{aligned}
\end{equation*}
where the $G_{1,\tau}$ and $G_{2,\tau}$ are the numerical solutions of $G_{1}$ and $G_{2}$ at the fixed time $t$ with step size $\tau$. The spatial and temporal convergence rates can be, respectively, calculated by
\begin{equation*}
	{\rm Rate}=\frac{\ln(E_{i,h}/E_{i,h/2})}{\ln(2)},\quad {\rm Rate}=\frac{\ln(E_{i,\tau}/E_{i,\tau/2})}{\ln(2)},\  i=1,2.
\end{equation*}
\subsection{One-dimensional cases}
\begin{example}
Consider the system \eqref{equmatrixequ} with the exact solution
\begin{equation*}
\begin{aligned}
G_1(x,t)=t^\nu x(1-x),\quad
G_2(x,t)=t^\nu x^2(1-x).
\end{aligned}
\end{equation*}
So the initial values are
\begin{equation*}
	\begin{aligned}
	G_{1,0}=0,\quad G_{2,0}=0,
	\end{aligned}
\end{equation*}
and the source terms
\begin{equation*}
	\begin{aligned}
	f_1(x,t)=&\nu t^{\nu-1}x(1-x)-2\frac{\Gamma(1+\nu)}{\Gamma(\nu+\alpha_1)}t^{\nu-1+\alpha_1}\\
	&+\frac{\Gamma(1+\nu)}{\Gamma(\nu+\alpha_1)}t^{\nu-1+\alpha_1}x(1-x)-\frac{\Gamma(1+\nu)}{\Gamma(\nu+\alpha_2)}t^{\nu-1+\alpha_2}x^2(1-x),\\
	f_2(x,t)=&\nu t^{\nu-1}x^2(1-x)-\frac{\Gamma(1+\nu)}{\Gamma(\nu+\alpha_2)}t^{\nu-1+\alpha_2}(2-6x)\\
	&+\frac{\Gamma(1+\nu)}{\Gamma(\nu+\alpha_2)}t^{\nu-1+\alpha_2}x^2(1-x)-\frac{\Gamma(1+\nu)}{\Gamma(\nu+\alpha_1)}t^{\nu-1+\alpha_1}x(1-x).\\
	\end{aligned}
\end{equation*}
Here we set $\nu=1.01$ and $a=2$. To get the spatial convergence rates, we take $\tau=0.1/1600$, so that the error incurred by temporal discretization is negligible, and the results are shown in Table \ref{tab:1dexam1spa}, which verify Theorem \ref{thmvanisheddatasemi}. Meanwhile, we take $h=1/256$ to get the temporal convergence rate, and Table \ref{tab:1dexam1time} shows the corresponding results, which validate Theorem \ref{thminhomfullest}.
\begin{table}[htbp]
	\caption{$L_2$ error at $t=0.1$}
	\label{tab:1dexam1spa}
		\begin{tabular}{c|c|cccc}
		\hline
		$(\alpha_1,\alpha_2)$&1/h            &          8 &         16 &         32 &         64 \\
		\hline
		&      $\|E_{1,h}\|_{L^2(\Omega)}$    &  2.417E-04 &  5.873E-05 &  1.316E-05 &  2.406E-06 \\
		
		(0.1,0.2) &            &    Rate  &    2.0409  &    2.1574  &    2.4519  \\
		
		&      $\|E_{2,h}\|_{L^2(\Omega)}$      &  2.694E-04 &  6.665E-05 &  1.585E-05 &  3.360E-06 \\
		
		&            &  Rate          &    2.0149  &    2.0719  &    2.2380  \\
		\hline
		&   $\|E_{1,h}\|_{L^2(\Omega)}$        &  2.232E-04 &  5.413E-05 &  1.210E-05 &  2.416E-06 \\
		
		(0.4,0.6) &            & Rate           &    2.0441  &    2.1620  &    2.3238  \\
		
		& $\|E_{2,h}\|_{L^2(\Omega)}$          &  2.563E-04 &  6.368E-05 &  1.548E-05 &  3.491E-06 \\
		
		&            &   Rate         &    2.0087  &    2.0404  &    2.1486  \\
		\hline
		&     $\|E_{1,h}\|_{L^2(\Omega)}$      &  1.924E-04 &  4.745E-05 &  1.142E-05 &  2.520E-06 \\
		
		(0.8,0.9) &            &Rate            &    2.0197  &    2.0551  &    2.1802  \\
		
		&  $\|E_{2,h}\|_{L^2(\Omega)}$         &  2.339E-04 &  5.831E-05 &  1.448E-05 &  3.528E-06 \\
		
		&            &  Rate         &    2.0040  &    2.0101  &    2.0370  \\
		\hline
	\end{tabular}
\end{table}

\begin{table}[htbp]
	\caption{$L_2$ error at $t=0.1$}
	\label{tab:1dexam1time}
	\begin{tabular}{c|c|ccccc}
		\hline
		$(\alpha_1,\alpha_2)$&$0.1/\tau$            &        100 &        200 &        400 &        800 &       1600 \\
		\hline
		&  $\|E_{1,\tau}\|_{L^2(\Omega)}$          &  3.389E-05 &  1.833E-05 &  9.758E-06 &  5.106E-06 &  2.612E-06 \\
		
		(0.1,0.2) &            & Rate           &    0.8869  &    0.9094  &    0.9342  &    0.9671  \\
		
		& $\|E_{2,\tau}\|_{L^2(\Omega)}$            &  2.375E-05 &  1.251E-05 &  6.480E-06 &  3.293E-06 &  1.628E-06 \\
		
		&            &  Rate          &    0.9251  &    0.9485  &    0.9765  &    1.0163  \\
		\hline
		&   $\|E_{1,\tau}\|_{L^2(\Omega)}$          &  4.254E-05 &  2.180E-05 &  1.105E-05 &  5.538E-06 &  2.727E-06 \\
		
		(0.4,0.6) &            &    Rate  &    0.9647  &    0.9797  &    0.9969  &    1.0219  \\
		
		&  $\|E_{2,\tau}\|_{L^2(\Omega)}$           &  1.512E-05 &  7.598E-06 &  3.767E-06 &  1.830E-06 &  8.611E-07 \\
		
		&            &  Rate         &    0.9930  &    1.0123  &    1.0417  &    1.0874  \\
		\hline
		&  $\|E_{1,\tau}\|_{L^2(\Omega)}$           &  1.538E-05 &  7.784E-06 &  3.923E-06 &  1.968E-06 &  9.795E-07 \\
		
		(0.8,0.9) &            & Rate           &    0.9824  &    0.9884  &    0.9955  &    1.0065  \\
		
		&  $\|E_{2,\tau}\|_{L^2(\Omega)}$           &  3.967E-06 &  2.025E-06 &  1.026E-06 &  5.133E-07 &  2.529E-07 \\
		
		&            &   Rate         &    0.9699  &    0.9817  &    0.9986  &    1.0214  \\
		\hline
	\end{tabular}
\end{table}

\end{example}
\begin{example}[Smooth initial value]
	Here consider the homogeneous problem \eqref{equmatrixequ} with smooth initial value, i.e.,
	\begin{equation*}
		G_{1,0}(x)=x(1-x),\quad G_{2,0}(x)=\sin(\pi x),
	\end{equation*}
	and $ f_1(x,t)=f_2(x,t)=0 $.
	It's easy to get that $G_{1,0}$, $G_{2,0}\in H^1_0(\Omega)\bigcap H^2(\Omega)$. Here, we choose $a=-10$. To investigate the convergence in space and eliminate the influence from temporal discretization, we take $\tau=0.01/1600$ and the results are shown in Table \ref{tab:1dexam2spa},  which verify Theorem \ref{thmsmoothdatasemi}. We take $h=1/256$ to verify the temporal convergence rate and the results are shown in Table \ref{tab:1dexam2time}, which agree with Theorem \ref{thmhomfullest}.
	\begin{table}[htbp]
		\caption{$L_2$ error at $t=0.01$}
		\label{tab:1dexam2spa}
		\begin{tabular}{c|c|cccc}
			\hline
			$(\alpha_1,\alpha_2)$&$1/h$            &          8 &         16 &         32 &         64 \\
			\hline
			&  $\|E_{1,h}\|_{L^2(\Omega)}$          &  7.045E-04 &  1.722E-04 &  4.281E-05 &  1.069E-05 \\
			
			(0.05,0.15) &            &  Rate          &    2.0324  &    2.0081  &    2.0020  \\
			
			& $\|E_{2,h}\|_{L^2(\Omega)}$           &  1.395E-03 &  3.444E-04 &  8.583E-05 &  2.144E-05 \\
			
			&            & Rate           &    2.0180  &    2.0045  &    2.0011  \\
			\hline
			&  $\|E_{1,h}\|_{L^2(\Omega)}$          &  4.233E-02 &  1.072E-02 &  2.689E-03 &  6.727E-04 \\
			
			(0.45,0.55) &            &Rate          0 &    1.9815  &    1.9954  &    1.9989  \\
			
			&  $\|E_{2,h}\|_{L^2(\Omega)}$          &  5.178E-02 &  1.309E-02 &  3.281E-03 &  8.209E-04 \\
			
			&            & Rate           &    1.9840  &    1.9960  &    1.9990  \\
			\hline
			&  $\|E_{1,h}\|_{L^2(\Omega)}$          &  1.673E-03 &  4.132E-04 &  1.030E-04 &  2.573E-05 \\
			
			(0.85,0.95) &            &Rate            &    2.0174  &    2.0044  &    2.0011  \\
			
			&  $\|E_{2,h}\|_{L^2(\Omega)}$          &  4.403E-03 &  1.086E-03 &  2.706E-04 &  6.758E-05 \\
			
			&            &  Rate          &    2.0195  &    2.0049  &    2.0012  \\
			\hline
		\end{tabular}
	\end{table}

	\begin{table}[htbp]
		\caption{$L_2$ error at $t=0.01$}
		\label{tab:1dexam2time}
		\begin{tabular}{c|c|ccccc}
			\hline
			$(\alpha_1,\alpha_2)$&$0.01/\tau$            &        100 &        200 &        400 &        800 &       1600 \\
			\hline
			&   $\|E_{1,\tau}\|_{L^2(\Omega)}$          &  3.032E-05 &  1.511E-05 &  7.541E-06 &  3.767E-06 &  1.883E-06 \\
			
			(0.05,0.15) &            &Rate            &    1.0049  &    1.0024  &    1.0012  &    1.0006  \\
			
			&  $\|E_{2,\tau}\|_{L^2(\Omega)}$           &  4.757E-05 &  2.369E-05 &  1.182E-05 &  5.903E-06 &  2.950E-06 \\
			
			&            &Rate            &    1.0061  &    1.0030  &    1.0016  &    1.0007  \\
			\hline
			&   $\|E_{1,\tau}\|_{L^2(\Omega)}$          &  1.088E-02 &  5.360E-03 &  2.660E-03 &  1.325E-03 &  6.613E-04 \\
			
			(0.45,0.55) &            &Rate          0 &    1.0215  &    1.0107  &    1.0053  &    1.0027  \\
			
			&  $\|E_{2,\tau}\|_{L^2(\Omega)}$           &  1.400E-02 &  6.900E-03 &  3.425E-03 &  1.706E-03 &  8.517E-04 \\
			
			&            & Rate           &    1.0210  &    1.0104  &    1.0052  &    1.0026  \\
			\hline
			&   $\|E_{1,\tau}\|_{L^2(\Omega)}$          &  1.753E-05 &  8.726E-06 &  4.354E-06 &  2.175E-06 &  1.087E-06 \\
			
			(0.85,0.95) &            &Rate          0 &    1.0061  &    1.0031  &    1.0015  &    1.0008  \\
			
			& $\|E_{2,\tau}\|_{L^2(\Omega)}$            &  5.946E-05 &  2.971E-05 &  1.485E-05 &  7.425E-06 &  3.712E-06 \\
			
			&            &  Rate          &    1.0009  &    1.0004  &    1.0002  &    1.0001  \\
			\hline
		\end{tabular}
		 	
	\end{table}
	Furthermore, we take the fixed $N$, $\alpha_1$, and $\alpha_2$ to validate Theorem \ref{thmhomfullest}'s estimates (theoretical decay rates with $t\rightarrow 0$)
	\begin{equation*}
		\begin{aligned}
		\|G_{1,h}(t_N)-G^N_{1,h}\|_{L^2(\Omega)}\leq C\left (N^{-1}\|G_{1,h}(0)\|_{L^2(\Omega)}+N^{-1}t_N^{\alpha_2}\|G_{2,h}(0)\|_{L^2(\Omega)}\right ),\\
		\|G_{2,h}(t_N)-G^N_{2,h}\|_{L^2(\Omega)}\leq C\left (N^{-1}t_N^{\alpha_1}\|G_{1,h}(0)\|_{L^2(\Omega)}+N^{-1}\|G_{2,h}(0)\|_{L^2(\Omega)}\right ).
		\end{aligned}
	\end{equation*}
	Let $N=10$, $\alpha_1=0.3$, $\alpha_2=0.7$, and define the initial values as
	\begin{equation*}
		G_{1,0}(x)=0,\quad G_{2,0}(x)=\sin(\pi x).
	\end{equation*}
	Thus the decay rates caused by $\|G_{1,h}(0)\|_{L^2(\Omega)}$ can be ignored and the theoretical decay rates of $\|G_{1,h}(t_N)-G^N_{1,h}\|_{L^2(\Omega)}$ and $ \|G_{2,h}(t_N)-G^N_{2,h}\|_{L^2(\Omega)} $  are $t^{\alpha_2}_N$ and $t^0_N$, respectively, when $t\rightarrow 0$. Table \ref{tab:1dexample2G10} shows that the temporal errors decrease like $t^{0.7}_N$ and $t^0_N$, respectively, when $t\rightarrow 0$.
	\begin{table}
		\caption{$L_2$ error as $t\rightarrow 0$}
		\label{tab:1dexample2G10}
		\begin{tabular}{c|cccccc}
			\hline
			&     1.E-01 &     1.E-02 &     1.E-03 &     1.E-04 &     1.E-05 &     1.E-06 \\
			\hline
			$\|E_{1,h}\|_{L^2(\Omega)}$&  4.540E-07 &  2.557E-07 &  7.693E-08 &  1.897E-08 &  4.293E-09 &  9.322E-10 \\
			
			&  Rate         &    0.2494  &    0.5216  &    0.6081  &    0.6452  &    0.6633  \\
			\hline
			$\|E_{2,h}\|_{L^2(\Omega)}$&  6.143E-06 &  1.179E-05 &  1.416E-05 &  1.509E-05 &  1.531E-05 &  1.536E-05 \\
			
			&    Rate        &   -0.2831  &   -0.0796  &   -0.0276  &   -0.0063  &   -0.0013  \\
			\hline
		\end{tabular}  	
	\end{table}
	Similarly, define the initial values as
	\begin{equation*}
	G_{1,0}(x)=x(1-x),\quad G_{2,0}(x)=0.
	\end{equation*}
	 Table \ref{tab:1dexample2G20} shows that the temporal errors decrease like $t^0_N$ and $t^{0.3}_N$ (when $t\rightarrow 0$),  respectively, which agree with the theoretical predictions.
	\begin{table}
			\caption{$L_2$ error as $t\rightarrow 0$}
		\label{tab:1dexample2G20}
		\begin{tabular}{c|cccccc}
			\hline
			&     1.E-01 &     1.E-02 &     1.E-03 &     1.E-04 &     1.E-05 &     1.E-06 \\
			\hline
			$\|E_{1,h}\|_{L^2(\Omega)}$&  8.957E-07 &  1.479E-06 &  2.157E-06 &  2.741E-06 &  3.175E-06 &  3.522E-06 \\
			
			&  Rate         &   -0.2178  &   -0.1640  &   -0.1041  &   -0.0637  &   -0.0451  \\
			\hline
			$\|E_{2,h}\|_{L^2(\Omega)}$&  2.236E-07 &  3.869E-07 &  3.071E-07 &  1.945E-07 &  1.124E-07 &  6.227E-08 \\
			
			&  Rate          &   -0.2382  &    0.1003  &    0.1983  &    0.2384  &    0.2563  \\
			\hline
		\end{tabular}
		
	\end{table}
\end{example}
\begin{example}[Nonsmooth initial value]
	Consider the homogeneous problem \eqref{equmatrixequ} with nonsmooth initial value. Let
	\begin{equation*}
	G_{1,0}(x)=\chi_{(3/4,1)}(x),\quad G_{2,0}(x)=\chi_{(0,1/4)}(x),
	\end{equation*}
	 $ f_1(x,t)=f_2(x,t)=0 $, and $a=10$. To validate the spatial convergence rates, we take $\tau=0.01/1600$ to  eliminate the influence from time discretization, and the results are shown in Table \ref{tab:1dexam3spa}, which verify Theorem \ref{thmnonsmoothdatasemi}. Then we let $h=1/256$ to get the temporal convergence rate, and Table \ref{tab:1dexam3time} provides the results, which verify Theorem \ref{thmhomfullest}.
	
	\begin{table}[htbp]
		\caption{$L_2$ error at t=0.01}
		\label{tab:1dexam3spa}
		\begin{tabular}{c|c|cccc}
		\hline
		($\alpha_1$,$\alpha_2$)&$1/h$            &          8 &         16 &         32 &         64 \\
		\hline
		&  $\|E_{1,h}\|_{L^2(\Omega)}$          &  7.125E-04 &  1.777E-04 &  4.439E-05 &  1.110E-05 \\
		
		(0.1,0.2) &            &  Rate          &    2.0037  &    2.0008  &    2.0002  \\
		
		&   $\|E_{2,h}\|_{L^2(\Omega)}$         &  1.030E-03 &  2.568E-04 &  6.416E-05 &  1.604E-05 \\
		
		&            &Rate           &    2.0038  &    2.0008  &    2.0002  \\
		\hline
		&   $\|E_{1,h}\|_{L^2(\Omega)}$         &  1.989E-03 &  4.959E-04 &  1.239E-04 &  3.097E-05 \\
		
		(0.4,0.6) &            & Rate           &    2.0042  &    2.0009  &    2.0002  \\
		
		&  $\|E_{2,h}\|_{L^2(\Omega)}$          &  3.609E-03 &  9.007E-04 &  2.251E-04 &  5.627E-05 \\
		
		&            &Rate            &    2.0025  &    2.0006  &    2.0002  \\
		\hline
		&  $\|E_{1,h}\|_{L^2(\Omega)}$          &  6.928E-03 &  1.742E-03 &  4.359E-04 &  1.090E-04 \\
		
		(0.8,0.9) &            &  Rate          &    1.9919  &    1.9986  &    1.9997  \\
		
		&  $\|E_{2,h}\|_{L^2(\Omega)}$          &  1.053E-02 &  2.674E-03 &  6.705E-04 &  1.677E-04 \\
		
		&            & Rate           &    1.9768  &    1.9958  &    1.9991  \\
		\hline
	\end{tabular}
	\end{table}
	
	\begin{table}[htbp]
		\caption{$L_2$ error at t=0.01}
		\label{tab:1dexam3time}
			\begin{tabular}{c|c|ccccc}
			\hline
			($\alpha_1$,$\alpha_2$)&$0.01/\tau$             &        100 &        200 &        400 &        800 &       1600 \\
			\hline
			&  $\|E_{1,\tau}\|_{L^2(\Omega)}$           &  8.156E-06 &  4.066E-06 &  2.030E-06 &  1.014E-06 &  5.069E-07 \\
			
			(0.1,0.2) &            &Rate            &    1.0043  &    1.0021  &    1.0011  &    1.0006  \\
			
			&  $\|E_{2,\tau}\|_{L^2(\Omega)}$           &  2.709E-05 &  1.349E-05 &  6.734E-06 &  3.364E-06 &  1.681E-06 \\
			
			&            &Rate            &    1.0055  &    1.0028  &    1.0014  &    1.0007  \\
			\hline
			&  $\|E_{1,\tau}\|_{L^2(\Omega)}$           &  7.198E-05 &  3.581E-05 &  1.786E-05 &  8.918E-06 &  4.456E-06 \\
			
			(0.4,0.6) &            &Rate            &    1.0073  &    1.0037  &    1.0018  &    1.0009  \\
			
			&  $\|E_{2,\tau}\|_{L^2(\Omega)}$           &  2.309E-04 &  1.150E-04 &  5.737E-05 &  2.866E-05 &  1.432E-05 \\
			
			&            &Rate            &    1.0060  &    1.0030  &    1.0015  &    1.0007  \\
			\hline
			&  $\|E_{1,\tau}\|_{L^2(\Omega)}$           &  5.587E-04 &  2.788E-04 &  1.392E-04 &  6.958E-05 &  3.478E-05 \\
			
			(0.8,0.9) &            &Rate            &    1.0030  &    1.0015  &    1.0008  &    1.0004  \\
			
			&  $\|E_{2,\tau}\|_{L^2(\Omega)}$           &  7.600E-04 &  3.804E-04 &  1.903E-04 &  9.516E-05 &  4.759E-05 \\
			
			&            &Rate            &    0.9986  &    0.9993  &    0.9997  &    0.9998  \\
			\hline
		\end{tabular}
	\end{table}
	
\end{example}
\subsection{Two-dimensional cases}
\begin{example}[Smooth initial data]
	Consider the two-dimensional homogeneous problem \eqref{equmatrixequ} with smooth initial value. Let
	\begin{equation*}
	G_{1,0}(x,y)=x(1-x)y(1-y),\quad G_{2,0}(x,y)=x^2(1-x)y(1-y)^2,
	\end{equation*}
	$ f_1(x,y,t)=f_2(x,y,t)=0 $, and $a=-2$. To get the spatial convergence rates, we take $\tau=0.1/1600$, so that the error incurred by temporal discretization is negligible, and the results are shown in Table \ref{tab:2dexam2spa}, which verify Theorem \ref{thmsmoothdatasemi}. Moreover, we let $h=1/256$ to obtain the temporal convergence rates, and the results are shown in Table \ref{tab:2dexam2time}, which confirm Theorem \ref{thmhomfullest}.
	\begin{table}[htbp]
		\caption{$L_2$ error at t=0.1}
		\label{tab:2dexam2spa}
		\begin{tabular}{c|c|cccc}
			\hline
			$(\alpha_1,\alpha_2)$&1/h            &          8 &         16 &         32 &         64 \\
			\hline
			&  $\|E_{1,h}\|_{L^2(\Omega)}$          &  4.194E-04 &  8.400E-05 &  1.856E-05 &  4.338E-06 \\
			
			(0.1,0.2) &            &          Rate &    2.3199  &    2.1779  &    2.0973  \\
			
			&  $\|E_{2,h}\|_{L^2(\Omega)}$          &  8.338E-05 &  1.548E-05 &  3.427E-06 &  8.032E-07 \\
			
			&            &         Rate   &    2.4290  &    2.1757  &    2.0931  \\
			\hline
			&  $\|E_{1,h}\|_{L^2(\Omega)}$          &  6.085E-04 &  1.220E-04 &  2.696E-05 &  6.300E-06 \\
			
			(0.4,0.6) &            &          Rate &    2.3186  &    2.1779  &    2.0973  \\
			
			& $\|E_{2,h}\|_{L^2(\Omega)}$           &  1.201E-04 &  2.251E-05 &  4.985E-06 &  1.168E-06 \\
			
			&            &          Rate  &    2.4162  &    2.1746  &    2.0930  \\
			\hline
			&  $\|E_{1,h}\|_{L^2(\Omega)}$          &  9.289E-04 &  1.983E-04 &  4.412E-05 &  1.032E-05 \\
			
			(0.8,0.9) &            &          Rate &    2.2276  &    2.1686  &    2.0956  \\
			
			&  $\|E_{2,h}\|_{L^2(\Omega)}$          &  9.814E-05 &  1.923E-05 &  4.290E-06 &  1.007E-06 \\
			
			&            &           Rate &    2.3513  &    2.1645  &    2.0912  \\
			\hline
		\end{tabular}	
	\end{table}
	\begin{table}[htbp]
	\caption{$L_2$ error at t=0.1}
	\label{tab:2dexam2time}
	\begin{tabular}{c|c|ccccc}
		\hline
		$(\alpha_1,\alpha_2)$&$0.1/\tau$            &         20 &         40 &         80 &        160 &        320 \\
		\hline
		&   $\|E_{1,\tau}\|_{L^2(\Omega)}$          &  5.400E-06 &  2.656E-06 &  1.317E-06 &  6.560E-07 &  3.273E-07 \\
		
		(0.1,0.2) &            &          Rate &    1.0236  &    1.0118  &    1.0059  &    1.0030  \\
		
		&   $\|E_{2,\tau}\|_{L^2(\Omega)}$          &  2.177E-06 &  1.066E-06 &  5.276E-07 &  2.624E-07 &  1.309E-07 \\
		
		&            & Rate           &    1.0301  &    1.0150  &    1.0075  &    1.0037  \\
		\hline
		&   $\|E_{1,\tau}\|_{L^2(\Omega)}$          &  4.007E-05 &  1.946E-05 &  9.590E-06 &  4.760E-06 &  2.372E-06 \\
		
		(0.4,0.6) &            &          Rate &    1.0420  &    1.0209  &    1.0104  &    1.0052  \\
		
		&  $\|E_{2,\tau}\|_{L^2(\Omega)}$           &  1.508E-05 &  7.211E-06 &  3.527E-06 &  1.744E-06 &  8.671E-07 \\
		
		&            &   Rate         &    1.0639  &    1.0320  &    1.0160  &    1.0080  \\
		\hline
		&  $\|E_{1,\tau}\|_{L^2(\Omega)}$           &  2.197E-04 &  1.069E-04 &  5.260E-05 &  2.608E-05 &  1.298E-05 \\
		
		(0.8,0.9) &            &          Rate &    1.0396  &    1.0226  &    1.0121  &    1.0062  \\
		
		& $\|E_{2,\tau}\|_{L^2(\Omega)}$            &  6.261E-05 &  2.999E-05 &  1.463E-05 &  7.221E-06 &  3.586E-06 \\
		
		&            &  Rate          &    1.0618  &    1.0354  &    1.0189  &    1.0098  \\
		\hline
	\end{tabular}
	
\end{table}
\end{example}
\begin{example}[Nonsmooth initial value]
	Consider the two-dimensional homogeneous problem \eqref{equmatrixequ} with nonsmooth initial value. Let
	\begin{equation*}
	G_{1,0}(x,y)=\chi_{(1/2,1)\times(0,3/4)}(x,y),\quad G_{2,0}(x,y)=\chi_{(0,3/4)\times(1/2,1)}(x,y),
	\end{equation*}
	$ f_1(x,y,t)=f_2(x,y,t)=0 $, and $a=1$. To get the spatial convergence rates, we take $\tau=0.1/1600$, so that the error incurred by temporal discretization is negligible, and the results are shown in Table \ref{tab:2dexam3spa}, which verify Theorem \ref{thmnonsmoothdatasemi}. At the same time, we let $h=1/256$ to obtain the temporal convergence rates and Table \ref{tab:2dexam3time} shows the results, which agree with Theorem \ref{thmhomfullest}.
		\begin{table}[htbp]
		\caption{$L_2$ error at t=0.1}
		\label{tab:2dexam3spa}
		\begin{tabular}{c|c|cccc}
			\hline
			$(\alpha_1,\alpha_2)$&    $1/h$        &          8 &         16 &         32 &         64 \\
			\hline
			&    $\|E_{1,h}\|_{L^2(\Omega)}$        &  5.973E-03 &  1.282E-03 &  2.847E-04 &  6.711E-05 \\
			
			(0.1,0.2) &            &          Rate &    2.2195  &    2.1714  &    2.0849  \\
			
			&  $\|E_{2,h}\|_{L^2(\Omega)}$          &  6.924E-03 &  1.477E-03 &  3.282E-04 &  7.737E-05 \\
			
			&            &          Rate  &    2.2287  &    2.1701  &    2.0849  \\
			\hline
			& $\|E_{1,h}\|_{L^2(\Omega)}$           &  8.642E-03 &  1.855E-03 &  4.114E-04 &  9.683E-05 \\
			
			(0.4,0.6) &            &          Rate &    2.2199  &    2.1728  &    2.0871  \\
			
			&  $\|E_{2,h}\|_{L^2(\Omega)}$          &  1.004E-02 &  2.141E-03 &  4.718E-04 &  1.095E-04 \\
			
			&            &          Rate  &    2.2288  &    2.1823  &    2.1069  \\
			\hline
			&   $\|E_{1,h}\|_{L^2(\Omega)}$         &  1.086E-02 &  2.342E-03 &  5.024E-04 &  1.102E-04 \\
			
			(0.8,0.9) &            &          Rate &    2.2129  &    2.2210  &    2.1884  \\
			
			&  $\|E_{2,h}\|_{L^2(\Omega)}$          &  1.164E-02 &  2.535E-03 &  5.328E-04 &  1.101E-04 \\
			
			&            &           Rate &    2.1988  &    2.2504  &    2.2755  \\
			\hline
		\end{tabular}
		
	\end{table}
\begin{table}[htbp]
	\caption{$L_2$ error at t=0.1}
	\label{tab:2dexam3time}
	\begin{tabular}{c|c|ccccc}
		\hline
		($\alpha_1$,$\alpha_2$)&$0.1/\tau$            &         20 &         40 &         80 &        160 &        320 \\
		\hline
		&   $\|E_{1,\tau}\|_{L^2(\Omega)}$          &  5.672E-05 &  2.790E-05 &  1.384E-05 &  6.889E-06 &  3.438E-06 \\
		
		(0.1,0.2) &            &         Rate   &    1.0236  &    1.0119  &    1.0059  &    1.0030  \\
		
		&  $\|E_{2,\tau}\|_{L^2(\Omega)}$           &  1.426E-04 &  6.982E-05 &  3.455E-05 &  1.719E-05 &  8.571E-06 \\
		
		&            &   Rate         &    1.0299  &    1.0149  &    1.0075  &    1.0037  \\
		\hline
		&  $\|E_{1,\tau}\|_{L^2(\Omega)}$           &  4.243E-04 &  2.060E-04 &  1.015E-04 &  5.038E-05 &  2.510E-05 \\
		
		(0.4,0.6) &            &        Rate    &    1.0426  &    1.0212  &    1.0106  &    1.0053  \\
		
		&$\|E_{2,\tau}\|_{L^2(\Omega)}$             &  9.640E-04 &  4.619E-04 &  2.261E-04 &  1.118E-04 &  5.563E-05 \\
		
		&            &   Rate         &    1.0614  &    1.0308  &    1.0154  &    1.0077  \\
		\hline
		& $\|E_{1,\tau}\|_{L^2(\Omega)}$            &  2.229E-03 &  1.072E-03 &  5.248E-04 &  2.596E-04 &  1.291E-04 \\
		
		(0.8,0.9) &            &      Rate      &    1.0562  &    1.0302  &    1.0156  &    1.0080  \\
		
		&  $\|E_{2,\tau}\|_{L^2(\Omega)}$           &  3.497E-03 &  1.718E-03 &  8.507E-04 &  4.231E-04 &  2.110E-04 \\
		
		&            & Rate           &    1.0252  &    1.0141  &    1.0075  &    1.0039  \\
		\hline
	\end{tabular}
\end{table}
\end{example}

\begin{example}
		Consider the inhomogeneous problem \eqref{equmatrixequ} with vanishing initial data. For checking the spatial convergence rates, to eliminate the influence from temporal discretization, we take $\tau=0.1/1600$ and set
	\begin{equation*}
	\begin{aligned}
			&f_1(x,y,t)=t^{0.2}xy,\quad G_{1,0}=0, \\
			&f_2(x,y,t)=t^{0.3},\quad G_{2,0}=0,
	\end{aligned}		
	\end{equation*}
	and $a=0.5$. Table \ref{tab:2dexample4G0spa}  provides the spatial convergence rates, which validate Theorem \ref{thmvanisheddatasemi}.
	\begin{table}[htbp]
		\caption{$L_2$ error at $t=0.1$}
		\label{tab:2dexample4G0spa}
		\begin{tabular}{c|c|cccc}
			\hline
			$(\alpha_1,\alpha_2)$&$1/h$            &          8 &         16 &         32 &         64 \\
			\hline
			&  $\|E_{1,h}\|_{L^2(\Omega)}$           &  5.452E-04 &  1.155E-04 &  2.591E-05 &  6.088E-06 \\
			
			(0.1,0.2) &            &          Rate &    2.2385  &    2.1566  &    2.0898  \\
			
			&  $\|E_{2,h}\|_{L^2(\Omega)}$          &  2.140E-03 &  4.341E-04 &  9.657E-05 &  2.266E-05 \\
			
			&            & Rate           &    2.3015  &    2.1682  &    2.0916  \\
			\hline
			& $\|E_{1,h}\|_{L^2(\Omega)}$           &  1.122E-03 &  2.383E-04 &  5.356E-05 &  1.259E-05 \\
			
			(0.4,0.6) &            &          Rate &    2.2355  &    2.1535  &    2.0886  \\
			
			& $\|E_{2,h}\|_{L^2(\Omega)}$           &  5.305E-03 &  1.072E-03 &  2.390E-04 &  5.612E-05 \\
			
			&            &   Rate         &    2.3064  &    2.1659  &    2.0903  \\
			\hline
			&  $\|E_{1,h}\|_{L^2(\Omega)}$          &  2.643E-03 &  5.625E-04 &  1.268E-04 &  2.984E-05 \\
			
			(0.8,0.9) &            &          Rate &    2.2321  &    2.1492  &    2.0871  \\
			
			& $\|E_{2,h}\|_{L^2(\Omega)}$           &  9.720E-03 &  1.948E-03 &  4.343E-04 &  1.021E-04 \\
			
			&            &    Rate        &    2.3190  &    2.1650  &    2.0895  \\
			\hline
		\end{tabular}
	\end{table}
For checking the temporal convergence rates, to eliminate the influence from
spatial discretization, we take $h=1/256$ and set
\begin{equation*}
\begin{aligned}
&f_1(x,y,t)=10t^{0.2}\chi_{(0,1/2)\times(1/4,1)}(x,y),\quad G_{1,0}=0, \\
&f_2(x,y,t)=10t^{0.3}\chi_{(1/2,1)\times(0,1/4)}(x,y),\quad G_{2,0}=0,
\end{aligned}		
\end{equation*}
$a=0.5$. The temporal convergence rates are shown in Table \ref{tab:2dexample4G0time}, which verify Theorem \ref{thminhomfullest}.
\begin{table}[htbp]
	\caption{$L_2$ error at $t=0.1$}
	\label{tab:2dexample4G0time}
	\begin{tabular}{c|c|ccccc}
		\hline
		$(\alpha_1,\alpha_2)$&$0.1/\tau$            &         80 &        160 &        320 &        640 &       1280 \\
		\hline
		&   $\|E_{1,\tau}\|_{L^2(\Omega)}$          &  3.355E-05 &  1.876E-05 &  1.025E-05 &  5.507E-06 &  2.920E-06 \\
		
		(0.1,0.2) &            &  Rate &    0.8387  &    0.8719  &    0.8966  &    0.9155  \\
		
		&  $\|E_{2,\tau}\|_{L^2(\Omega)}$           &  5.092E-06 &  2.950E-06 &  1.642E-06 &  8.897E-07 &  4.731E-07 \\
		
		&            & Rate &    0.7878  &    0.8452  &    0.8839  &    0.9111  \\
		\hline
		&  $\|E_{1,\tau}\|_{L^2(\Omega)}$           &  3.147E-05 &  1.290E-05 &  5.202E-06 &  2.053E-06 &  7.877E-07 \\
		
		(0.4,0.6) &            & Rate &    1.2863  &    1.3106  &    1.3410  &    1.3821  \\
		
		&  $\|E_{2,\tau}\|_{L^2(\Omega)}$           &  2.300E-05 &  1.092E-05 &  5.217E-06 &  2.510E-06 &  1.214E-06 \\
		
		&            & Rate &    1.0753  &    1.0652  &    1.0559  &    1.0475  \\
		\hline
		&  $\|E_{1,\tau}\|_{L^2(\Omega)}$           &  2.791E-04 &  1.358E-04 &  6.627E-05 &  3.244E-05 &  1.592E-05 \\
		
		(0.8,0.9) &            & Rate &    1.0395  &    1.0349  &    1.0308  &    1.0272  \\
		
		&  $\|E_{2,\tau}\|_{L^2(\Omega)}$           &  8.463E-05 &  4.176E-05 &  2.064E-05 &  1.022E-05 &  5.066E-06 \\
		
		&            & Rate &    1.0192  &    1.0167  &    1.0143  &    1.0121  \\
		\hline
	\end{tabular}
\end{table}
\end{example}
\section*{Conclusion}

Anomalous diffusions are ubiquitous in natural world. The models are built for describing the different types of anomalous diffusions. The more recent FFPEs with multiple internal states effectively characterize the anomalous diffusion with different waiting time distributions for different internal states, governing the distribution of positions of the particles.
In this paper, we develop the Sobolev regularity of the FFPEs, including the homogeneous problem with smooth and nonsmooth initial values and the inhomogeneous problem with vanishing initial value, and then we design a numerical scheme for the FFPEs based on the finite element approximation for the space derivatives and convolution quadrature for the time fractional derivatives.
We provide the optimal error estimates for the schemes in different cases, including the space semidiscrete and fully discrete schemes.
Finally, the numerical experiments for one- and two-dimensional examples are performed to confirm the theoretical analyses and the predicted convergence orders.



\end{document}